\documentclass[11pt,a4paper]{article}

\usepackage[utf8]{inputenc}
\usepackage[T1]{fontenc}
\usepackage{amsmath}
\usepackage{amsfonts}
\usepackage{amssymb}
\usepackage{amsthm}
\usepackage{thmtools}
\usepackage{enumerate}
\usepackage{mathtools} 
\usepackage{xfrac}
\usepackage{lmodern}
\usepackage[nottoc]{tocbibind} 
\usepackage{textgreek}
\usepackage{graphicx}
\usepackage[dvipsnames]{xcolor} 
\usepackage[english]{babel}
\usepackage{slashed}
\usepackage{microtype}
\usepackage[autostyle=true]{csquotes}
\usepackage[all,cmtip]{xy}
\usepackage{tikz-cd}
\usepackage{rotating}

\definecolor{Heather}{RGB}{164, 132, 172}
\usepackage[pdfusetitle,bookmarks,pdfpagelabels]{hyperref} 
\hypersetup{
  colorlinks,
  urlcolor=Periwinkle,
  citecolor=Heather,
  linkcolor=teal,
  breaklinks=true,
  final,
}

\usepackage[a4paper,
            top=2.5cm,
            bottom=2.5cm,
            inner=3.5cm,
            outer=3.5cm]{geometry}

\usepackage[style=alphabetic,
            backend=biber,
            url=true,
            doi=true,
            isbn=false,
            giveninits=true]{biblatex}

\DeclareDelimFormat[textcite]{multinamedelim}{\textendash}
\DeclareDelimAlias[textcite]{finalnamedelim}{multinamedelim}

\DeclareSourcemap{
  \maps{
    \map{ %
        \step[fieldset = month, null]
        \step[fieldset = day, null]
        \step[fieldset = language, null]
    }
    \map{ %
        \step[fieldsource = doi, final]
        \step[fieldset = url, null]
    }
    \map{ %
        \step[fieldsource = eprint, final]
        \step[fieldset = url, null]
    }
  }
}

\usepackage[capitalise]{cleveref} 

\usepackage{xparse} 
\usepackage{enumitem}
\setlist[enumerate]{label={\upshape(\roman*)}}
\newlist{myenumi}{enumerate}{1}
\setlist[myenumi,1]{label=\upshape(\roman*)}
\newlist{myenuma}{enumerate}{1}
\setlist[myenuma,1]{label=\upshape(\alph*)}

\declaretheorem[name=Theorem,numberwithin=section]{thm}
\declaretheorem[name=Lemma,numberlike=thm]{lem}
\declaretheorem[name=Corollary,numberlike=thm]{cor}
\declaretheorem[name=Proposition,numberlike=thm]{prop}

\declaretheorem[name=Definition,numberlike=thm,style=definition]{defn}

\declaretheorem[name=Remark,numberlike=thm,style=definition]{rem}

\crefdefaultlabelformat{#2\textup{#1}#3}
\crefname{thm}{Theorem}{Theorems}
\crefname{lem}{Lemma}{Lemmas}
\crefname{defn}{Definition}{Definitions}
\crefname{prop}{Proposition}{Propositions}
\crefname{cor}{Corollary}{Corollaries}
\crefname{equation}{}{}

\newcommand{\N}{\mathbb{N}}

\newcommand{\Z}{\mathbb{Z}}

\newcommand{\R}{\mathbb{R}}

\newcommand{\C}{\mathbb{C}}

\newcommand{\vol}{\mathrm{vol}}
\newcommand{\Ric}{\mathrm{Ric}}

\newcommand{\EM}{\mathrm{EM}}
\newcommand{\BC}{\mathrm{BC}}

\newcommand{\cX}{\mathcal{X}}
\newcommand{\cY}{\mathcal{Y}}
\newcommand{\cA}{\mathcal{A}}
\newcommand{\cU}{\mathcal{U}}

\newcommand{\cK}{\mathcal{K}}
\newcommand{\cB}{\mathcal{B}}

\newcommand{\Mult}{\mathcal{M}}

\newcommand{\KK}{\mathrm{K\!\, K}}

\newcommand{\Bfree}{\mathrm{B}}
\newcommand{\Efree}{\mathrm{E}}
\newcommand{\Eub}{\underline{\mathrm{E}}}
\newcommand{\Ct}{\mathrm{C}}
\newcommand*{\Cz}{\Ct_0}

\newcommand*{\Cb}{\Ct_{\mathrm{b}}}

\newcommand*{\Lin}{\cB}
\newcommand*{\Kom}{\cK}
\newcommand*{\U}{\cU}
\newcommand*{\elltwo}{\ell^2}

\DeclareMathOperator{\Hom}{Hom}

\newcommand{\K}{\mathrm{K}}
\newcommand{\KO}{\mathrm{KO}}
\newcommand{\Ktop}{\mathrm{K}^\mathrm{top}}
\newcommand{\Strg}{\mathrm{S}}
\newcommand{\StrgReal}{\mathrm{S}\mathbf{R}}

\newcommand*{\StolzPos}{\mathrm{Pos}^{\mathrm{spin}}}
\newcommand*{\SpinBordism}{\Omega^{\mathrm{spin}}}
\newcommand*{\StolzRel}{\mathrm{R}^{\mathrm{spin}}}

\newcommand*{\BottMfd}{\mathrm{Bt}}

\newcommand*{\sHigComRed}{\bar{\mathfrak{c}}^{\mathrm{red}}}
\newcommand*{\sHigCorRed}{\mathfrak{c}^{\mathrm{red}}}

\newcommand*{\Ball}{\mathrm{B}}
\newcommand{\clBall}{\overline{\mathrm{B}}}

\newcommand*{\red}{\mathrm{red}}
\NewDocumentCommand{\Cstar}{}{\ensuremath{\mathrm{C}^*}}
\NewDocumentCommand{\textCstar}{}{\ensuremath{\mathrm{C}^*\!}}
\NewDocumentCommand{\CstarRed}{}{\Cstar_{\red}}

\NewDocumentCommand{\LSym}{}{\mathrm{L}}

\NewDocumentCommand{\D}{}{\mathop{}\!\mathrm{d}}

\NewDocumentCommand \RoeSymbol {o} {
	\mathrm{C}^{\ast}
	\IfNoValueF{#1}{_{#1}}
}

\NewDocumentCommand \VanishSymbol {o} {
	\mathrm{N}^{\ast}
	\IfNoValueF{#1}{_{#1}}
}

\NewDocumentCommand \FiproSymbol {o} {
	\mathrm{E}^{\ast}
	\IfNoValueF{#1}{_{#1}}
}

\NewDocumentCommand \RoePlaceholder {o} {
\RoeSymbol[
	\IfNoValueF{#1}{#1,}
	\mathrm{?}
]
}

\NewDocumentCommand \Roe {o} {\RoeSymbol[#1]}
\NewDocumentCommand \Fipro {o} {\FiproSymbol[#1]}

\NewDocumentCommand \varRoe {o} {
	\RoeSymbol[\sim\IfNoValueF{#1}{,#1}]
}

\NewDocumentCommand \Loc {o} {
	\RoeSymbol[
		\IfNoValueF{#1}{#1,}
		\LSym
	]
}

\NewDocumentCommand \LocVanish {o} {
	\VanishSymbol[
		\IfNoValueF{#1}{#1,}
    \LSym
	]
}

\NewDocumentCommand \FiproLoc {o} {
	\FiproSymbol[
		\IfNoValueF{#1}{#1,}
		\LSym
	]
}

\NewDocumentCommand \varLoc {o} {
	\RoeSymbol[
		\sim,
		\IfNoValueF{#1}{#1,}
		\LSym
	]
}

\NewDocumentCommand \Locz {o} {
	\RoeSymbol[
		\IfNoValueF{#1}{#1,}
		\LSym,0
	]
}

\NewDocumentCommand \FiproLocz {o} {
	\FiproSymbol[
		\IfNoValueF{#1}{#1,}
		\LSym,0
	]
}

\NewDocumentCommand \varLocz {o} {
	\RoeSymbol[
		\sim,
		\IfNoValueF{#1}{#1,}
		\LSym,0
	]
}

\DeclareMathOperator{\im}{im}
\newcommand{\id}{\mathrm{id}}

\DeclareMathOperator{\Var}{\operatorname{Var}}

\DeclareMathOperator{\Ind}{Ind}

\DeclareMathOperator{\ev}{ev}

\NewDocumentCommand{\blank}{}{{-}}

\newcommand*{\tensmax}{\mathbin{\otimes_{\max}}}
\newcommand*{\tens}{\otimes}

\numberwithin{equation}{section} 

\author{Christopher Wulff\thanks{
Mathematisches Institut, Georg--August--Universität Göttingen, Bunsenstraße 3--5, 37073 Göttingen, Germany\newline
\href{mailto:christopher.wulff@mathematik.uni-goettingen.de}{christopher.wulff@mathematik.uni-goettingen.de}}
\and Rudolf Zeidler\thanks{Universität Potsdam, Institut für Mathematik, Karl-Liebknecht-Str.~24--25, 14476 Potsdam, Germany\newline
\href{mailto:rudolf.zeidler@uni-potsdam.de}{rudolf.zeidler@uni-potsdam.de}}
}

\title{Duality pairings with the analytic structure group}
\date{}

\begin{document}

\maketitle

\begin{abstract}
  We construct a slant product $\mathrm{S}^{G\times H}_p(X\times Y)\otimes \mathrm{K}_{-q}(\bar{\mathfrak{c}}^{\mathrm{red}} Y\rtimes H)\to \mathrm{K}_{p-q}(\mathrm{C}^\ast_G X)$ on the analytic structure group of Higson and Roe and the K-theory of the stable Higson compactification taking values in the (equivariant) Roe algebra.
  This complements the slant products constructed in earlier work of Engel and the authors ({\tt arXiv:1909.03777 [math.KT]}).
  The distinguishing feature of our new slant product is that it specializes to a duality pairing $\mathrm{S}^H_p(Y) \otimes \mathrm{K}_{-p}(\bar{\mathfrak{c}}^{\mathrm{red}} (Y)\rtimes H)\to \mathbb{Z}$ which can be used to extract numerical invariants out of elements in the analytic structure group such as rho-invariants associated to positive scalar curvature metrics.
\end{abstract}

\setcounter{tocdepth}{2}
\tableofcontents

\section{Introduction}
The analytic exact sequence of Higson and Roe~\cite{MappingSurgeryToAnalysisIII} is the target for primary and secondary index invariants which arise, for instance, in surgery theory or in the study of metrics with positive scalar curvature (see \cite{PiazzaSchick:StolzPSC,XieYuPscLocAlg,zeidler_secondary} for the latter). 
In \cite{EngelWulffZeidler}, Engel and the authors constructed and analyzed in depth natural slant products on the Higson--Roe exact sequence of the form
\[\xymatrix{
\Strg_p^{G\times H}(X \times Y) \ar[r] \ar[d]^{/ \theta} & \K_p^{G \times H}(X \times Y) \ar[r] \ar[d]^{/ \theta} & \K_p(\Roe[G \times H](X \times Y)) \ar[d]^{/ \theta} \ar[r]^-{\partial} & \Strg_{p-1}^{G \times H}(X \times Y) \ar[d]^{/ \theta}\\
\Strg_{p-q}^G(X) \ar[r] & \K_{p-q}^G(X) \ar[r] & \K_{p-q}(\Roe[G] X) \ar[r]^-{\partial} & \K_{p-1-q}^G(X).
}\]
for proper metric spaces $X,Y$, where $Y$ has bounded geometry, endowed with proper isometric actions of countable discrete groups $G$ and $H$, respectively. 
The slant products are by arbitrary elements $\theta \in \K_{1-q}(\sHigCorRed Y \rtimes_{\mu} H)$, that is, the $\K$-theory of the crossed product of the stable Higson corona $\sHigCorRed Y$ by the induced $H$-action.
Here, as in \cite{EngelWulffZeidler}, \enquote{\(\rtimes_{\mu}\)} denotes any exact crossed product functor in the sense of \cite[Definition~3.1]{BaumGuentWillExactCrossed}.
For instance, we may use the maximal crossed product.

These slant products are dual to the exterior products
\begin{align*}
  \Strg_p^{G}(X) \otimes \K^H_q(Y) &\to \Strg_{p+q}^{G \times H}(X \times Y), \\
  \K_{p}(\Roe[G](X)) \otimes \K_{q}(\Roe[H](Y)) &\to \K_{p+q}(\Roe[G \times H](X \times Y)),
\end{align*}
and were designed to detect non-vanishing of elements in the Higson--Roe exact sequence, in particular of secondary invariants in the structure group which arise as a product, and of primary index invariants in the K-theory of the equivariant Roe algebra \(\Roe[G](X)\).
The slant product involving the Roe algebra can be specialized to a pairing
\begin{equation*}
  \K_{p}(\Roe[H](Y)) \otimes \K_{-p}(\sHigCorRed Y \rtimes_\mu H) \to \K_{0}(\C) = \Z, 
\end{equation*}
which goes back to \textcite{EmeMey}.
In other words, elements of the group \(\K_{-\ast}(\sHigCorRed X \rtimes_\mu G)\) can be viewed as numerical invariants defined on the K-theory of the equivariant Roe algebra \(\K_{\ast}(\Roe[G](X))\).

The motivation of the present article is to construct such a duality pairing which involves the analytic structure group and thereby obtain numerical invariants on the analytic structure group.
To this end, we amend the picture of \cite{EngelWulffZeidler} by constructing another slant product
\begin{equation}\Strg^{G\times H}_p(X\times Y)\otimes\K_{-q}(\sHigComRed Y\rtimes_{\mu} H)\to\K_{p-q}(\Roe[G]X) \label{eq:new_slant}\end{equation}
between the analytic structure group and the $\K$-theory of the stable Higson compactification (rather than the stable Higson \emph{corona} which was involved in the slant product of \cite{EngelWulffZeidler}), taking values in the K-theory of a Roe algebra.
The slant product \labelcref{eq:new_slant} then specializes to a duality pairing
\[
 \Strg^H_p(Y) \otimes \K_{-p}(\sHigComRed (Y)\rtimes_\mu H)\to \K_{0}(\C) = \Z
\]
which can be used to define numerical invariants on the analytic structure group (see \cref{sec:numerical}) which, in turn, can be applied to extract numerical secondary invariants associated to positive scalar curvature metrics (see \cref{sec:PSC_application}).

This abstract construction raises the question if suitable elements of \(\K_{\ast+1}(\sHigComRed(\mathcal{X}) \rtimes_\mu G)\) can be found to obtain interesting numerical invariants on the analytic structure group.
In many cases this reduces to a topological question because the group \(\K_{-p}(\sHigComRed (Y)\rtimes_\mu H)\) can be related to a certain relative \(\K\)-theory group.
Indeed, as we discuss in \cref{sec:numerical}, there is a canonical map 
\begin{equation*}
  \delta_{Y} \colon \K_{-p}(\sHigComRed(Y) \rtimes_\mu H) \to \K_H^{p+1}(\Eub H, Y),
  \end{equation*}
  where \(\Eub H\) is the classifying space for proper \(H\)-actions.
  Moreover, this map \(\delta_Y\) is surjective if \(H\) admits an \(H\)-finite \(\Eub H\) and has a \(\gamma\)-element, see~\cref{prop:dual_boundary_surj}.
  However, we emphasize that the slant product \labelcref{eq:new_slant} exists unconditionally and it remains an intriguing question in which situations interesting elements of \(\K_{-p}(\sHigComRed(Y) \rtimes_\mu H)\) can be found independently of assumptions on the Baum--Connes assembly map (such as the existence of a \(\gamma\)-element).

Constructing the product \labelcref{eq:new_slant} requires, in comparison to the one of \cite{EngelWulffZeidler}, the somewhat stronger assumption of continuously bounded geometry on $Y$ which on complete Riemannian manifolds is implied by a lower Ricci bound, see \Cref{defn:contboundedGeometry,prop:ricci_implies_cbg}. 
Moreover, we also discuss relative versions of the slant products, which had not been done in \cite{EngelWulffZeidler}, and show their compatibility with each other.

\paragraph{Acknowledgements.}
We are grateful for the support by the SPP 2026 \enquote{Geometry at Infinity}.
Funded by the Deutsche Forschungsgemeinschaft (DFG, German Research Foundation) – Project numbers
338480246; %
390685587; %
427320536; %
523079177. %

\section{The Higson--Roe exact sequence}

We start by recalling the construction of the Higson--Roe exact sequence of a proper metric space $X$, using the notation of \cite[Section 3.1]{EngelWulffZeidler}.

By an $X$-module we mean a pair $(H_X,\rho_X)$ consisting of a separable Hilbert space $H_X$ together with a non-degenerate $*$-representation $\rho_X\colon\Cz(X)\to\Lin(H_X)$.
The associated \emph{Roe algebra} $\Roe(\rho_X)$ is defined as the sub-\textCstar-algebra of $\Lin(H_X)$ generated by all locally compact operators of finite propagation.

Here we say an operator $T \in \Lin(H_X)$ is \emph{locally compact} if $\rho_X(f) T$ and $T \rho_X(f)$ are compact operators for all $f \in \Cz(X)$.
Moreover, \(T \in \Lin(H_X)\) has \emph{finite propagation} if there exists an $R > 0$ such that $\rho_X(f) T \rho_X(g) = 0$ provided that the supports of $f$ and $g$ have distance greater than \(R\) from each other.
An \(X\)-module \((H_X, \rho_X)\) is called \emph{ample} if only the zero function in $\Cz(X)$ acts by a compact operator on \(H_X\) via \(\rho_X\).
 From now on we shall assume that we have fixed an ample \(X\)-module \((H_X, \rho_X)\).
 Note that we usually suppress the \(X\)-module from the notation because the K-theory of Roe algebras does not depend on the choice of ample \(X\)-module up to canonical isomorphism \cites[Corollary~6.3.13]{higson_roe}[Theorem~5.1.15]{WillettYuHigherIndexTheory}.
 
 Next, we de define the \emph{localization algebra} $\Loc X$ to be the sub-\Cstar-algebra of $\Cb([1,\infty),\Roe X)$ generated by the bounded and uniformly continuous functions $L\colon [1,\infty) \to\Roe X$ such that the propagation
of $L(t)$ is finite for all $t\geq 1$ and tends to zero as $t\to\infty$.
For ample \(X\)-modules its $\K$-theory is canonically isomorphic to the locally finite $\K$-homology of the space $X$, see e.g.~\cite{yu_localization,qiao_roe} and \cite[Chapters~6--7]{WillettYuHigherIndexTheory}.
Let $\Locz X$ be the ideal in \(\Loc X\) consisting of those \(L \in \Loc X\) with \(L(1) = 0\).
We define the \emph{analytic structure group of $X$} as the \(\K\)-theory of \(\Locz X\), denoted by $\Strg_\ast(X)\coloneqq\K_\ast(\Locz X)$.

By construction, we obtain a short exact sequence of \textCstar\nobreakdash-algebras,
\[0\to\Locz X\to\Loc X\xrightarrow{\operatorname{ev}_1}\Roe X\to 0\]
whose induced long exact sequence is called the \emph{Higson--Roe sequence}
\[\dots\to \K_{*+1}(\Roe X)\to \Strg_*(X)\to \K_*(X)\xrightarrow{\Ind} \K_*(\Roe X)\to\dots\]
where the index map $\Ind=(\operatorname{ev}_1)_*$ is induced by evaluation at $1$.

\textcite{EmeMey} constructed a dual version of the Higson--Roe sequence which we briefly recall now.
A function $f\colon Y\to Z$ from a proper metric space $Y$ into another metric space $Z$ is said to have \emph{vanishing variation} if the function
\[\Var_r f\colon Y\to[0,\infty]\,,\quad x\mapsto\sup\{d(f(x),f(y))\mid y\in B_r(x)\}\]
converges to zero at infinity for all $r>0$.
The \emph{reduced stable Higson compactification} $\sHigComRed(Y)$ is then defined as the \textCstar-algebra of all bounded continuous functions $f\colon Y\to\Lin(\ell^2)$ of vanishing variation such that $f(x)-f(y)\in\Kom$ for all $x,y\in Y$, where $\Kom\coloneqq\Kom(\ell^2)$ denotes the \textCstar-algebra of compact operators on the standard separable infinite dimensional Hilbert space $\ell^2\coloneqq\ell^2(\N)$. It contains $\Cz(Y,\Kom)$ as an ideal and one defines the \emph{reduced stable Higson corona} as the quotient $\sHigCorRed(Y)\coloneqq\sHigComRed(Y)/\Cz(Y,\Kom)$.
We immediately obtain a long exact sequence
\[\dots \to \K_{1-*}(\sHigCorRed Y)\xrightarrow{\mu^*}\K^*(Y)\to\K_{-*}(\sHigComRed Y)\to\K_{-*}(\sHigCorRed Y)\to\dots\]
whose connecting homomorphism $\mu^*$ is called the coarse co-assembly map.

\section{Construction of the non-equivariant slant products}
\label{sec:constr_noneq_slant}

From now on, let $X,Y$ be proper metric spaces. 
Furthermore, we assume that $Y$ has continuously bounded geometry, which is a crucial prerequisite for our constructions of the slant products on the structure groups and hence we make it a standing assumption.
Recall that this property is defined as follows.

\begin{defn}[{Definition 4.1(b) in \cite{EngelWulffZeidler}}]\label{defn:contboundedGeometry}
A metric space $Y$ is said to have \emph{continuously bounded geometry} if for every $r>0$ and \(R > 0\) there exists a constant $K_{r,R}>0$ such that the following conditions hold.
\begin{itemize}
  \item For every \(r > 0\) there is a subset $\hat{Y}_r\subset Y$ such that $Y=\bigcup_{\hat y\in \hat{Y}_r}\Ball_r(\hat y)$ and such that for all $r,R>0$ and \(y \in Y\) the number $\#(\hat{Y}_r\cap \clBall_R(y))$ is bounded by $K_{r,R}$.
  \item For all \(\alpha>0\), we have \(K_\alpha\coloneqq\limsup_{r\to 0}K_{r,\alpha r}<\infty\).
\end{itemize}
\end{defn}

This is a stronger property than the usual bounded geometry of metric spaces. The slant products on $\K$-homology and on the $\K$-theory of the Roe algebra can be defined under the assumption of the usual notion of bounded geometry of metric spaces, but for the definition of the slant products on the structure group we require the stronger continuously bounded geometry.

In contrast, the usual notion of bounded geometry in Riemannian geometry is stronger than having continuously bounded geometry.
In fact, a lower bound on the Ricci curvature suffices as we establish in the following proposition.
The proof is modelled on the proof of Gromov's precompactness theorem, compare e.g.~\cite[Corollary~11.1.13]{Petersen:Riemannian}.
\begin{prop}\label{prop:ricci_implies_cbg}
  Let \(Y\) be a complete Riemannian \(n\)-manifold with a uniform Ricci curvature lower bound \(\Ric \geq \kappa (n-1)\) for some \(\kappa \in \R\).
  Then \(Y\) has continuously bounded geometry.
\end{prop}
\begin{proof}
  Without loss of generality and to simplify some formulas, we may assume that \(\kappa = -1\).
  Let \(v_\rho > 0\) for \(\rho > 0\) denote the volume of a \(\rho\)-ball in hyperbolic \(n\)-space.
  Then the Bishop--Gromov relative volume comparison theorem (see e.g.~\cite[Lemma~7.1.4]{Petersen:Riemannian}) states that for every \(y \in Y\) the function
  \[(0,\infty) \to (0,\infty), \quad \rho \mapsto \frac{\vol(\Ball_\rho(y))}{v_\rho}\]
  is non-increasing and it tends to \(1\) as \(\rho \searrow 0\).
  
  For each \(r > 0\) let \(\hat{Y}_r \subset Y\) be a maximal subset such that the family of balls \(\Ball_{r/2}(\hat{y})\), \(\hat{y} \in \hat{Y}\), is pairwise disjoint.
  Then by maximality, \(Y=\bigcup_{\hat y\in \hat{Y}_r}\Ball_r(\hat y)\).
  
  Now fix \(R > 0\), \(r > 0\) and \(y \in Y\).
  Let \(\hat{y}_1, \dotsc, \hat{y}_N \in \hat{Y}_r \cap \clBall_R(y)\) be any collection of distinct points for some \(N \in \N\).
  Relabel these points so that \(\vol(\Ball_{r/2}(\hat{y}_1))\) is minimal among all \(\vol(\Ball_{r/2}(\hat{y}_i))\) for \(i \in \{1, \dotsc, N\}\).
  Note that for all \(x \in \Ball_{r/2}(\hat{y}_i)\), we have
  \[
	d(x, \hat{y}_1) \leq d(x,\hat{y}_i) + d(\hat{y}_i,y) + d(y,\hat{y}_1) < r/2 + 2R	  
\]
  and thus
  \(\Ball_{r/2}(\hat{y}_i) \subseteq \Ball_{2 R + r/2}(\hat{y}_1)\) for all \(i \in \{1, \dotsc, N\}\).
  Since the balls \(\Ball_{r/2}(\hat{y}_i)\) are pairwise disjoint and \(\vol(\Ball_{r/2}(\hat{y}_1)) \leq \vol(\Ball_{r/2}(\hat{y}_i))\) for all \(i \in \{1, \dotsc, N\}\), we obtain
  \[
	  N \leq \frac{\vol(\Ball_{2R + r/2}(\hat{y_1}))}{\vol(\Ball_{r/2}(\hat{y}_1))} \leq \frac{v_{2R+r/2}}{v_{r/2}},
  \]
  where we have used the Bishop--Gromov volume comparison for the second inequality.
  We conclude that
  \[
	  \#(\hat{Y}_r \cap \clBall_R(y)) \leq \frac{v_{2R+r/2}}{v_{r/2}} \eqqcolon K_{r,R}.
	  \]
	  
 We explicitly have \(v_\rho = \omega_{n-1} \int_{0}^\rho \sinh(t)^{n-1} \D{t}\) for \(\omega_{n-1} > 0\) the volume of the \((n-1)\)-sphere, and a concrete calculation (apply de L'Hospital twice) yields
	\[
	\lim_{r\to 0} K_{r, \alpha r} = \lim_{r\to 0} \frac{v_{r (2\alpha  + 1/2)}}{v_{r/2}} = (4 \alpha + 1) \lim_{r\to 0} \left(\frac{\sinh(r (2\alpha  + 1/2))}{\sinh(r/2)} \right)^{n-1} = (4 \alpha +1)^n < \infty
	\]
	for every \(\alpha > 0\) and \(r > 0\).
	Thus \(Y\) has continuously bounded geometry.
\end{proof}

Let us review two slant products from \cite{EngelWulffZeidler}, before we define our new slant product. The first is the slant product between the $\K$-theory of the Roe algebra and the $\K$-theory of the stable Higson corona (\Cref{defn_coarseslant} below) and the second one is the slant product between $\K$-homology and $\K$-theory of spaces (\Cref{defn_alternativeKhomslant} below).

We fix an ample $X$-module $(H_X,\rho_X)$ and an ample $Y$-module $(H_Y,\rho_Y)$. Then $H_{X\times Y}\coloneqq H_X\otimes H_Y$ together with the tensor product representation $\rho_{X\times Y}\coloneqq\rho_X\otimes\rho_Y\colon\Cz(X\times Y)\to\Lin(H_{X\times Y})$ is an ample $X\times Y$-module.
Furthermore, we define another ample $X$-module $(\tilde H_X,\tilde\rho_X)$ by $\tilde H_X\coloneqq H_X\otimes H_Y\otimes \ell^2$ and 
\begin{equation}\label{eq_varRepOfCzX}
\tilde\rho_X\coloneqq\rho_X\otimes\id_{H_Y\otimes\ell^2}\colon\Cz(X)\to\Lin(\tilde H_X)\,.
\end{equation}

Using the tensor product of $\rho_Y$ and the canonical representation of $\Kom$ on $\elltwo$, we obtain a non-degenerate representation of $\Cz(Y, \Kom)$ on $H_Y \otimes \elltwo$ which, in turn, extends to a strictly continuous representation of the multiplier algebra
\begin{equation}\label{eq_rep_of_Cb}
\bar\rho_Y\colon \Mult(\Cz(Y, \Kom))\to \cB(H_Y \otimes \ell^2).
\end{equation}
Then we define the representation
\begin{equation*}
\tilde\rho_Y\coloneqq \id_{H_X}\otimes \bar\rho_Y\colon\Mult(\Cz(Y, \Kom)) \to \cB(\tilde H_X)
\end{equation*}
whose image commutes with the image of the representation $\tilde\rho_X$.
Moreover, there is a canonical embedding of $\Cb(Y, \Kom)$ into $\Mult(\Cz(Y, \Kom))$.
We also define the representation 
\[\tau\colon\Roe(\rho_{X\times Y})\to\Lin(\tilde H_X)\,,\quad S\mapsto S\otimes\id_{\ell^2}\,.\]

Now, let $\Fipro(\tilde\rho_X)\subset\Lin(\tilde H_X)$ be the sub-\textCstar-algebra generated by all finite propagation operators. We summarize the technical results from Section 4.1.1 of \cite{EngelWulffZeidler}.
\begin{lem}[{Lemmas 4.2, 4.3, 4.5, 4.6 in \cite{EngelWulffZeidler}}]\ 
\label{lem:summarylemma}
\begin{enumerate}
\item The Roe algebra $\Roe(\tilde\rho_X)$ is an ideal in $\Fipro(\tilde\rho_X)$.
\item The images of the representations $\tau$ and $\tilde\rho_Y$ defined above are contained in $\Fipro(\tilde\rho_X)$.
\item The image of the representation $\tau$ commutes up to $\Roe(\tilde\rho_X)$ with the image of $\sHigComRed Y$ under the representation $\tilde\rho_Y$. Hence, by the universal property of the maximal tensor product, there is an induced $*$-homomorphism
\begin{equation*}
\Phi\colon\Roe(\rho_{X\times Y})\otimes_{\max} \sHigComRed Y\to \Fipro(\tilde\rho_X)/\Roe(\tilde\rho_X)
\end{equation*}
determined by $S\otimes f\mapsto [\tau(S)\circ\tilde\rho_Y(f)]$.

\item The $*$-homomorphism $\Phi$ factors through the tensor product $\Roe(\rho_{X \times Y}) \otimes_{\max} \sHigCorRed Y$.
In other words, it defines a $*$-homomorphism
\begin{equation*}
\Psi\colon\Roe(\rho_{X \times Y}) \otimes_{\max} \sHigCorRed Y \to  \Fipro(\tilde\rho_X)/\Roe(\tilde\rho_X).
\end{equation*}
\end{enumerate}\qed
\end{lem}

\begin{defn}[{Definition 4.7 in \cite{EngelWulffZeidler}}]\label{defn_coarseslant}
The slant product between the $\K$-theory of the Roe algebra and the $\K$-theory of the reduced stable Higson corona is now defined as $(-1)^p$ times the composition
\begin{align*}
\K_p(\Roe(X \times Y)) \otimes \K_{1-q}(\sHigCorRed Y) \ &= \K_p(\Roe(\rho_{X \times Y})) \otimes \K_{1-q}(\sHigCorRed Y)
\\& \xrightarrow{\boxtimes} \ \K_{p+1-q}(\Roe(\rho_{X \times Y}) \otimes_{\max} \sHigCorRed Y)
\\& \xrightarrow{\Psi_*} \ \K_{p+1-q}(\Fipro(\tilde\rho_X)/\Roe(\tilde\rho_X))
\\& \xrightarrow{\partial} \ \K_{p-q}(\Roe(\tilde\rho_X))
\\& = \K_{p-q}(\Roe X),
\end{align*}
where the first arrow is the external product on $\K$-theory, and the third arrow the boundary operator in the corresponding long exact sequence.
\end{defn}

The second slant product that we want to recall is the one between $\K$-homology and $\K$-theory of spaces. For all second countable locally compact Hausdorff spaces it can simply be defined by means of the products in $\KK$-theory or E-theory, but for proper metric spaces of continuously bounded geometry, there is a description in terms of localization algebras which is more suitable for our purposes.

\begin{lem}[{cf.\ Corollary 4.18 and the subsequent paragraph in \cite{EngelWulffZeidler}}]
Suppose that $Y$ has continuously bounded geometry.
Let $T \in \Loc(\rho_{X \times Y})$ and $f \in \sHigComRed Y$.
Then
\begin{equation*}
	t \mapsto \left[ \tau(T_t), \tilde\rho_Y(f) \right] \in \Cz([1,\infty), \Roe(\tilde\rho_X))
\end{equation*}
and if \(f \in \Cz(Y,\Kom)\), then also
\begin{equation*}
	t \mapsto  \tau(T_t)\circ \tilde\rho_Y(f) \in \Loc(\tilde\rho_X)\,.
\end{equation*}
Hence there is a well-defined $*$-homomorphism
\begin{align*}
\Upsilon_{\LSym} \colon \Loc(\rho_{X\times Y}) \otimes \Cz(Y, \Kom) & \to \Loc(\tilde\rho_X) / \Cz([1,\infty),\Roe(\tilde\rho_X)),\label{eq_defn_Ypsilon}\\
(T_t)_{t \in [1,\infty)} \otimes f & \mapsto \left[ (\tau(T_t) \circ \tilde\rho_Y(f))_{t \in [1,\infty)} \right].\notag
\end{align*}
\qed
\end{lem}
Furthermore, let $\ev_\infty\colon\Loc(\tilde\rho_X) \to \Loc(\tilde\rho_X) / \Cz([1,\infty), \Roe(\tilde\rho_X))$ be the canonical quotient $*$-homomorphism, which we also think of as an ``evaluation at infinity''. It induces an isomorphism on $\K$-theory, because its kernel is contractible.

Now, the slant product from the next definition corresponds exactly to the $E$-theoretic slant product by Lemma 4.19 in \cite{EngelWulffZeidler}.
\begin{defn}\label{defn_alternativeKhomslant}
We define the slant product between $\K$-homology and $\K$-theory of proper metric spaces $X,Y$ with $Y$ having continuously bounded geometry as the composition
\begin{align*}
	\K_p(X \times Y) \otimes \K^{q}(Y)
    &\cong \K_p(\Loc(\rho_{X \times Y})) \otimes \K_{-q}(\Cz(Y, \Kom)) \\
  &\xrightarrow{\boxtimes} \K_{p-q}(\Loc(\rho_{X\times Y}) \otimes \Cz(Y, \Kom)) \\
	&\xrightarrow{({\Upsilon_{\LSym}})_\ast} \K_{p-q}( \Loc(\tilde\rho_X) / \Cz([1, \infty), \Roe(\tilde\rho_X))) \\
	&\xrightarrow[\cong]{(\ev_\infty)^{-1}} \K_{p-q}(\Loc(\tilde\rho_X)) \cong \K_{p-q}(X)\,.
\end{align*}
\end{defn}

We can now proceed to the definition of our new slant product. 
The relevant constructions can also be found in Section 4 of \cite{EngelWulffZeidler} and we briefly review them.
Let $\FiproLoc(\tilde{\rho}_X)$ be the $\Cstar$-subalgebra of $\Cb([1,\infty),\Fipro(\tilde{\rho}_X))$ generated by the bounded and uniformly continuous functions $S\colon [1,\infty) \to \Fipro(\tilde{\rho}_X)$ such that the propagation of $S(t)$ is finite for all $t \ge 1$ and tends to zero as $t \to \infty$.
Similarly, let $\FiproLocz(\tilde{\rho}_X)$ denote the ideal in $\FiproLoc(\tilde{\rho}_X)$ consisting of functions vanishing at \(1\).
Then $\Loc(\tilde{\rho}_X)$ is an ideal in $\FiproLoc(\tilde{\rho}_X)$.
Moreover, $\Locz(\tilde{\rho}_X)$ is an ideal in each of $\FiproLoc(\tilde{\rho}_X)$, $\FiproLocz(\tilde{\rho}_X)$ and $\Loc(\tilde{\rho}_X)$.
\color{black}

\begin{lem}[{cf.\ proof of Theorem 4.13 in \cite{EngelWulffZeidler}}]
The $*$-homomorphism $\Upsilon_{\LSym}$ extends to a $*$-homomorphism
\begin{align*}
\bar{\Upsilon}_{\LSym} \colon \Loc(\rho_{X \times Y}) \tensmax \sHigComRed Y & \to \FiproLoc(\tilde\rho_X) / \Cz([1,\infty),\Roe(\tilde\rho_X)),\\
(T_t)_{t \in [1,\infty)} \otimes f & \mapsto \left[ (\tau(T_t) \circ \tilde\rho_Y(f))_{t \in [1,\infty)} \right].
\end{align*}
This extension clearly restricts to a $*$-homomorphism
\[\bar{\Upsilon}_{\LSym,0} \colon \Locz(\rho_{X \times Y}) \tensmax \sHigComRed Y \to \FiproLocz(\tilde\rho_X) / \Cz((1,\infty),\Roe(\tilde\rho_X))\,.\]
\qed
\end{lem}

\begin{defn}\label{defn_structuregroupslant}
The slant product between the structure group and the $\K$-theory of the stable Higson compactification is defined as the negative\footnote{The heuristic reason for this sign is that the two boundary maps appearing in the composition correspond once to the boundary $1$ of $[1,\infty)$ and once to the boundary $\infty$ of $(1,\infty]$, so loosely speaking they should correspond to the negative of one another. This idea will become precise in the proof of \Cref{lem:structuregroupslantcompatiblewithhomologycohomologyslant}.} of the composition
\begin{align*}
\Strg_p(X\times Y)\otimes\K_{-q}(\sHigComRed Y)&\cong \K_p(\Locz(\rho_{X\times Y})) \otimes\K_{-q}(\sHigComRed Y)
\\& \xrightarrow{\boxtimes} \K_{p-q}(\Locz(\rho_{X \times Y}) \otimes_{\max} \sHigComRed Y)
\\& \xrightarrow{(\bar{\Upsilon}_{\LSym,0})_*} \K_{p-q}(\FiproLocz(\tilde\rho_X) / \Cz((1,\infty),\Roe(\tilde\rho_X)))
\\& \xrightarrow{\partial} \K_{p-q-1}( \Cz((1,\infty),\Roe(\tilde\rho_X)))
\\&\xrightarrow[\cong]{\partial^{-1}}\K_{p-q}(\Roe(\tilde\rho_X))
= \K_{p-q}(\Roe X)
\end{align*}
where the two boundary maps are the ones associated to the short exact sequences
\[\xymatrix@C=3ex@R=1ex{
0\ar[r]&\Cz((1,\infty),\Roe(\tilde\rho_X))\ar[r]\ar@{=}[d]&\FiproLocz(\tilde\rho_X)\ar[r]&\frac{\FiproLocz(\tilde\rho_X)}{\Cz((1,\infty),\Roe(\tilde\rho_X))}\ar[r]&0
\\0\ar[r]&\Cz((1,\infty),\Roe(\tilde\rho_X))\ar[r]&\Cz([1,\infty),\Roe(\tilde\rho_X))\ar[r]&\Roe(\tilde\rho_X)\ar[r]&0
}\]
of \textCstar-algebras.
\end{defn}

\begin{lem}\label{lem_LocSlantNatural}
The slant product between the structure group and the $\K$-theory of the stable Higson compactification is up to canonical isomorphism independent of the choice of ample modules and natural under pairs of continuous coarse maps $\alpha\colon X\to X'$ and $\beta\colon Y\to Y'$ in the sense that
\begin{equation*}%
\alpha_*(x/\beta^*(\theta))=(\alpha\times \beta)_*(x)/\theta
\end{equation*}
for all $x\in \Strg_*(X\times Y)$ and $\theta\in \K_*(\sHigComRed Y')$.
\end{lem}

\begin{proof}
The naturality under uniformly continuous coarse maps is proven completely analogously to \cite[Theorem 4.31]{EngelWulffZeidler} along the lines of \cite[Theorem 4.28]{EngelWulffZeidler}. The generalization to continuous coarse maps works with the trick presented in \cite[Remark 4.32]{EngelWulffZeidler}.

Independence of the representations follows by applying the naturality for the identity maps $\alpha=\id_X,\beta=\id_Y$, if one considers different ample modules over source and target.
\end{proof}

We can now check the compatibility of this slant product with the other two under the maps of the Higson--Roe exact sequence and its dual.

\begin{lem}\label{lem:structuregroupslantcompatiblewithRoestableHigsoncoronaslant}
The diagram
\[\xymatrix@C=10ex{
\K_p(\Roe(X\times Y))\otimes\K_{1-q}(\sHigComRed Y)\ar[r]^-{\partial\otimes\id}\ar[d]
&\Strg_{p-1}(X\times Y)\otimes\K_{1-q}(\sHigComRed Y)\ar[d]^-{/}
\\\K_p(\Roe(X\times Y))\otimes\K_{1-q}(\sHigCorRed Y)\ar[r]^-{/}
&\K_{p-q}(\Roe(X))
}\]
commutes up to the sign $(-1)^p$.
\end{lem}

\begin{proof}
Consider the commutative diagrams of \textCstar-algebras
\[\xymatrix{
&0\ar[d]&0\ar[d]&0\ar[d]&
\\0\ar[r]&\Cz((1,\infty),\Roe(\tilde\rho_X))\ar[d]\ar[r]&\Cz([1,\infty),\Roe(\tilde\rho_X))\ar[d]\ar[r]&\Roe(\tilde\rho_X)\ar[d]\ar[r]&0
\\0\ar[r]&\FiproLocz(\tilde\rho_X)\ar[d]\ar[r]&\FiproLoc(\tilde\rho_X)\ar[d]\ar[r]&\Fipro(\tilde\rho_X)\ar[d]\ar[r]&0
\\0\ar[r]&\frac{\FiproLocz(\tilde\rho_X)}{\Cz((1,\infty),\Roe(\tilde\rho_X))}\ar[d]\ar[r]&\frac{\FiproLoc(\tilde\rho_X)}{\Cz([1,\infty),\Roe(\tilde\rho_X))}\ar[d]\ar[r]&\frac{\Fipro(\tilde\rho_X)}{\Roe(\tilde\rho_X)}\ar[d]\ar[r]&0
\\&0&0&0&
}\]
with exact rows and columns and
\[\xymatrix@C=2ex{
0\ar[r]&\Locz(\rho_{X\times Y})\otimes_{\max}\sHigComRed Y\ar[d]^{\bar{\Upsilon}_{\LSym,0}}\ar[r]&\Loc(\rho_{X\times Y})\otimes_{\max}\sHigComRed Y\ar[d]^{\bar{\Upsilon}_{\LSym}}\ar[r]&\Roe(\rho_{X\times Y})\otimes_{\max}\sHigComRed Y\ar[d]^{\Phi}\ar[r]&0
\\0\ar[r]&\frac{\FiproLocz(\tilde\rho_X)}{\Cz((1,\infty),\Roe(\tilde\rho_X))}\ar[r]&\frac{\FiproLoc(\tilde\rho_X)}{\Cz([1,\infty),\Roe(\tilde\rho_X))}\ar[r]&\frac{\Fipro(\tilde\rho_X)}{\Roe(\tilde\rho_X)}\ar[r]&0
}\]
with exact rows. 
Together with the external product, they induce on $\K$-theory the diagram 
\[\xymatrix@C=8ex{
\K_p(\Roe(\rho_{X\times Y}))\otimes\K_{1-q}(\sHigComRed Y)\ar[d]^{\boxtimes}\ar[r]^-{\partial\otimes\id}
&\K_{p-1}(\Locz(\rho_{X\times Y}))\otimes\K_{1-q}(\sHigComRed Y)\ar[d]^{\boxtimes}
\\\K_{p+1-q}(\Roe(\rho_{X\times Y})\otimes_{\max}\sHigComRed Y)\ar[d]^-{\Phi_*}\ar[r]^-{\partial}
&\K_{p-q}(\Locz(\rho_{X\times Y})\otimes_{\max}\sHigComRed Y)\ar[d]^{(\bar{\Upsilon}_{\LSym,0})_*}
\\\K_{p+1-q}\left(\frac{\Fipro(\tilde\rho_X)}{\Roe(\tilde\rho_X)}\right)\ar[r]^-{\partial}\ar[d]^-{\partial}
&\K_{p-q}\left(\frac{\FiproLocz(\tilde\rho_X)}{\Cz((1,\infty),\Roe(\tilde\rho_X))}\right)\ar[d]^-{\partial}
\\\K_{p-q}(\Roe(\tilde\rho_X))\ar[r]^-{\partial}_-{\cong}&\K_{p-q-1}(\Cz((1,\infty),\Roe(\tilde\rho_X)))
}\]
whose upper two squares commute and whose lower square commutes up to the sign $-1$.
The arrows on the left hand side compose to $(-1)^p$ times
\[\K_p(\Roe(\rho_{X\times Y}))\otimes\K_{1-q}(\sHigComRed Y)\to \K_p(\Roe(\rho_{X\times Y}))\otimes\K_{1-q}(\sHigCorRed Y)\xrightarrow{/}\K_{p-q}(\Roe(\tilde\rho_X))\]
and the arrows on the right and bottom compose to the negative of the slant product 
\[\K_{p-1}(\Locz(\rho_{X\times Y}))\otimes\K_{1-q}(\sHigComRed Y)\xrightarrow{/} \K_{p-q}(\Roe(\tilde\rho_X))\,.\]
The claim follows.
\end{proof}

\begin{lem}\label{lem:structuregroupslantcompatiblewithhomologycohomologyslant}
The diagram
\[\xymatrix{
\K_p(X\times Y)\otimes\K^q(Y)\ar[d]^{/}&\Strg_p(X\times Y)\otimes\K^q(Y)\ar[l]\ar[r]&\Strg_p(X\times Y)\otimes\K_{-q}(\sHigComRed Y)\ar[d]^{/}
\\\K_{p-q}(X)\ar[rr]^-{\Ind}&&\K_{p-q}(\Roe(X))
}\]
commutes. 
\end{lem}
\begin{proof}
We will deduce the claim from the diagram in \Cref{fig:proofstructuregroupslantcompatiblewithhomologycohomologyslant},
\begin{sidewaysfigure}
\[\xymatrix@C=10ex@R=10ex{
\K_p(\Loc(\rho_{X\times Y}))\otimes\K_{-q}(\Cz(Y))\ar[d]^-{\boxtimes}
&\K_p(\Locz(\rho_{X\times Y}))\otimes\K_{-q}(\Cz(Y))\ar[d]^-{\boxtimes}\ar[l]\ar[r]
&\K_p(\Locz(\rho_{X\times Y}))\otimes\K_{-q}(\sHigComRed Y)\ar[d]^-{\boxtimes}
\\\K_{p-q}(\Loc(\rho_{X\times Y})\otimes\Cz(Y))\ar[d]^-{(\Upsilon_{\LSym})_*}
&\K_{p-q}(\Locz(\rho_{X\times Y})\otimes\Cz(Y))\ar[d]^-{(\Upsilon_{\LSym,0})_*}\ar[l]\ar[r]
&\K_{p-q}(\Locz(\rho_{X\times Y})\otimes_{\max}\sHigComRed Y)\ar[d]^-{(\bar\Upsilon_{\LSym,0})_*}
\\\K_{p-q}\left(\frac{\Loc(\tilde\rho_X)}{\Cz([1,\infty),\Roe(\tilde\rho_X))}\right)\ar[dd]^{\cong}_{(\ev_\infty)_*^{-1}}
&\K_{p-q}\left(\frac{\Locz(\tilde\rho_X)}{\Cz((1,\infty),\Roe(\tilde\rho_X))}\right)\ar[l]^\cong\ar[r]\ar[dr]^{\partial}
&\K_{p-q}\left(\frac{\FiproLocz(\tilde\rho_X)}{\Cz((1,\infty),\Roe(\tilde\rho_X))}\right)\ar[d]^{\partial}
\\
&
&\K_{p-q-1}(\Cz((1,\infty),\Roe(\tilde\rho_X)))\ar[d]^{\cong}_{\partial^{-1}}
\\\K_{p-q}(\Loc(\tilde\rho_X))\ar[rr]^{\Ind=(\ev_1)_*}
&
&\K_{p-q}(\Roe(\tilde\rho_X))
}\]
\caption{Proving Lemma \ref{lem:structuregroupslantcompatiblewithhomologycohomologyslant}.}
\label{fig:proofstructuregroupslantcompatiblewithhomologycohomologyslant}
\end{sidewaysfigure}
in which $\Upsilon_{\LSym,0}$ denotes the obvious simultaneous restriction of both $\Upsilon_{\LSym}$ and $\bar\Upsilon_{\LSym,0}$.
The left and right vertical compositions are the two slant products up to the signs $+1$ and $-1$, respectively.
Commutativity of the whole diagram except of the lower pentagon is clear and we just need to show that the latter commutes up to a sign $-1$.

Consider the diagram of \textCstar-algebras
\[\xymatrix{
0\ar[r]&\Cz((1,\infty),\Roe(\tilde\rho_X))\ar[r]\ar@{=}[d]
&\Locz(\tilde\rho_X)\ar[r]\ar[d]%
&\frac{\Locz(\tilde\rho_X)}{\Cz((1,\infty),\Roe(\tilde\rho_X))}\ar[d]^{\iota_\infty}\ar[r]%
&0
\\0\ar[r]&\Cz((1,\infty),\Roe(\tilde\rho_X))\ar[r]\ar@{=}[d]
&\Loc(\tilde\rho_X)\ar[r]^{\pi}\ar@{-->}[ur]_{\ev_\infty}\ar@{-->}[dr]^{\ev_1}
&\frac{\Loc(\tilde\rho_X)}{\Cz((1,\infty),\Roe(\tilde\rho_X))}\ar[r]&0
\\0\ar[r]&\Cz((1,\infty),\Roe(\tilde\rho_X))\ar[r]
&\Cz([1,\infty),\Roe(\tilde\rho_X))\ar[r]\ar[u]
&\Roe(\tilde\rho_X)\ar[u]_{\iota_1}\ar[r]&0
}\]
whose rows are exact and whose solid arrows commute.
Here, $\ev_\infty$ has been identified with the composition $\Loc(\tilde\rho_X)\xrightarrow{\ev_\infty} \frac{\Loc(\tilde\rho_X)}{\Cz([1,\infty),\Roe(\tilde\rho_X))}\cong\frac{\Locz(\tilde\rho_X)}{\Cz((1,\infty),\Roe(\tilde\rho_X))}$.
The dashed arrows do not commute with the rest, but instead $\iota_\infty\circ\ev_\infty$ and $\iota_1\circ\ev_1$ are two $*$-homomorphisms with commuting images and whose sum is the quotient $*$-homomorphism $\pi$.

Thus, we obtain on $\K$-theory the diagram
\[\xymatrix{
&\K_{p-q}(\frac{\Locz(\tilde\rho_X)}{\Cz((1,\infty),\Roe(\tilde\rho_X))})\ar[d]^{(\iota_\infty)_*}\ar[dr]^-{\partial}%
&
\\\K_{p-q}(\Loc(\tilde\rho_X))\ar[r]^-{\pi_*}\ar@{-->}[ur]^-{(\ev_\infty)_*}_-{\cong}\ar@{-->}[dr]_-{\Ind=(\ev_1)_*}
&\K_{p-q}(\frac{\Loc(\tilde\rho_X)}{\Cz((1,\infty),\Roe(\tilde\rho_X))})\ar[r]^-{\partial}
&\K_{p-q-1}(\Cz((1,\infty),\Roe(\tilde\rho_X)))
\\&\K_{p-q}(\Roe(\tilde\rho_X))\ar[u]_{(\iota_1)_*}\ar[ur]^-{\partial}_-{\cong}
}\]
whose right part commutes and such that $\pi_*=(\iota_\infty)_*\circ(\ev_\infty)_*+(\iota_1)_*\circ(\ev_1)_*$ as well as $\partial\circ\pi_*=0$. Therefore, 
\begin{align*}
0&=\partial\circ\pi_*
=\partial\circ((\iota_\infty)_*\circ(\ev_\infty)_*+(\iota_1)_*\circ(\ev_1)_*)
=\partial\circ(\ev_\infty)_*+\partial\circ(\ev_1)_*
\end{align*}
and the claim follows.
\end{proof}

Furthermore, we have the following theorem describing the compatibility between our new slant product and the external products introduced in \cite[Section 3.1]{EngelWulffZeidler}.
\begin{lem}\label{lem:structuregroupslantscompatiblewithcrossproducts}
Let $X,Y,Z$ be proper metric spaces, $Z$ of continuously bounded geometry, and $m,p,q\in\Z$.
Then the following diagram commutes:
\[\xymatrix{
\K_m(X)\otimes\Strg_p(Y\times Z)\otimes\K_{-q}(\sHigComRed Z)\ar[r]^-{\id\otimes /}\ar[d]^{\times\otimes\id}&\K_m(X)\otimes\K_{p-q}(\Roe Y)\ar[d]^{\times\circ(\Ind\otimes \id)}
\\\Strg_{m+p}(X\times Y\times Z)\otimes \K_{-q}(\sHigComRed Z)\ar[r]^-{/}&\K_{m+p-q}(\Roe(X\times Y))
}\]
\end{lem}

\begin{proof}
The proof is very similar to the proof of \cite[Theorem 4.20]{EngelWulffZeidler}. The essential observation is that using suitable representations yields a commutative diagram:
\[\xymatrix@C=10ex{
\Loc(X)\tensmax\Locz(Y\times Z)\tensmax\sHigComRed Z
\ar[r]^-{\id\otimes\bar\Upsilon_{\LSym,0}}\ar[d]
&\Loc(X)\tensmax\frac{\FiproLocz(Y)}{\Cz((1,\infty),\Roe(Y))}\ar[d]
\\\Locz(X\times Y\times Z)\tensmax\sHigComRed Z\ar[r]^-{\bar\Upsilon_{\LSym,0}}
&\frac{\FiproLocz(X\times Y)}{\Cz((1,\infty),\Roe(X\times Y))}
}\]
Here, the right vertical map applied to $(T_t)_{t\in[1,\infty)}\otimes\blank$ depends on all $T_t$ and not just $T_1$, but this dependence will disappear once we apply the last two maps $\partial^{-1}\circ\partial$ in the definition of our slant product, when the ideal $\Cz((1,\infty),\Roe(X\times Y))$ is divided out in the second short exact sequence. This is where the map $\Ind=(\ev_1)$ comes into play.
\end{proof}

\section{The relative case}

The slant products from the previous section have relative versions as well. 
We briefly recall the definitions of relative Roe and localization algebras as well as relative stable Higson coronas and compactifications.
Their most basic properties can be proven exactly as in the absolute case and hence we just refer the reader to the standard references, e.\,g.\ \cite{higson_roe,WillettYuHigherIndexTheory}.
See also \cite[Sections 4.3,4.4,6.3,6.4]{wulff2020equivariant}, where this has already been discussed even in the equivariant case and with coefficient-\textCstar-algebras.

Let $A\subset X$ and $B\subset Y$ be closed subspaces.
First of all, we define the relative reduced stable Higson compactification and corona of $(Y,B)$ as the kernels of the restriction to $B$, that is,
\begin{align*}
\sHigComRed(Y,B)&\coloneqq\ker(\sHigComRed(Y)\to\sHigComRed(B))
\\\sHigCorRed(Y,B)&\coloneqq\ker(\sHigCorRed(Y)\to\sHigCorRed(B))=\sHigComRed(Y,B)/\Cz(Y\setminus B,\Kom)\,.
\end{align*}
Note that on $\K$-theory we get the following diagram with exact rows and columns where all squares commute except of the ones bounded only by boundary maps, which commute up to the sign $-1$:
\begin{equation}\label{eq:RelativeDualHigsonRoeDiagram}
\xymatrix{
&\vdots\ar[d]&\vdots\ar[d]&\vdots\ar[d]&
\\\dots\ar[r]&\K^{-*}(Y,B)\ar[r]\ar[d]&\K^{-*}(Y)\ar[r]\ar[d]&\K^{-*}(B)\ar[r]\ar[d]&\dots
\\\dots\ar[r]&\K_*(\sHigComRed(Y,B))\ar[r]\ar[d]&\K_*(\sHigComRed(Y))\ar[r]\ar[d]&\K_*(\sHigComRed(B))\ar[r]\ar[d]&\dots
\\\dots\ar[r]&\K_*(\sHigCorRed(Y,B))\ar[r]\ar[d]&\K_*(\sHigCorRed(Y))\ar[r]\ar[d]&\K_*(\sHigCorRed(B))\ar[r]\ar[d]&\dots
\\&\vdots&\vdots&\vdots&
}\end{equation}

Dually one defines the relative Roe and localization algebras as follows. There are ideals $\Roe(A\subset X)\subset\Roe(X)$ and $\Fipro(A\subset X)\subset\Fipro(X)$ which are defined as the norm closure of all those operators in the respective \textCstar-algebra whose support is contained in some $R$-neighborhood of $A\times A$. Clearly, $\Roe(A\subset X)=\Roe(X)\cap\Fipro(A\subset X)$ is also an ideal in $\Fipro(A\subset X)$ and we define the relative Roe algebras as the quotients
\[\Roe(X,A)\coloneqq \Roe(X)/\Roe(A\subset X)\quad\text{and}\quad \Fipro(X,A)\coloneqq \Fipro(X)/\Fipro(A\subset X)\,,\]
where $\Roe(X,A)\cong \frac{\Roe(X)+\Fipro(A\subset X)}{\Fipro(A\subset X)}$ embeds canonically as an ideal into $\Fipro(X,A)$.

Similarly, the ideals $\Loc(A\subset X)\subset\Loc(X)$, $\Locz(A\subset X)\subset\Locz(X)$, $\FiproLoc(A\subset X)\subset\FiproLoc(X)$ and $\FiproLocz(A\subset X)\subset\FiproLocz(X)$ are defined as the norm closure of all those families of operators $(T_t)_{t\in[1,\infty)}$ in the respective \textCstar-algebras for which there is a function $\varepsilon\colon [1,\infty)\to(0,\infty)$ with $\varepsilon(t)\xrightarrow{t\to\infty}0$ such that $T_t$ is supported  in an $\varepsilon(t)$-neighborhood of $A\times A$ for all $t\geq 1$. We define the four relative localization algebras
\begin{align*}
\Loc(X,A)&\coloneqq \Loc(X)/\Loc(A\subset X)
&\Locz(X,A)&\coloneqq \Locz(X)/\Locz(A\subset X)
\\\FiproLoc(X,A)&\coloneqq \FiproLoc(X)/\FiproLoc(A\subset X)
&\FiproLocz(X,A)&\coloneqq \FiproLocz(X)/\FiproLocz(A\subset X)\,.
\end{align*}

Here, we already omitted the ample $X$-module $(H_X,\rho_X)$ from the notation to keep the formulas legible. The justification of this is, again, that the $\K$-theory groups $\K_*(\Loc(X,A))\cong \K_*(X,A)$, $\Strg_*(X,A)\coloneqq\K_*(\Locz(X,A))$ and $\K_*(\Roe(X,A))$ do not depend on its particular choice up to canonical isomorphism.
These are the relative $\K$-homology, structure group and $\K$-theory of the Roe algebra for which we will consider slant products. They are covariantly functorial under continuous, continuous coarse and coarse maps, respectively, between pairs of proper metric spaces $(X,A)\to (X',A')$, and this functoriality is again implemented by adjoining with suitable (families of) covering isometries.

It is furthermore well-known that there are canonical natural isomorphisms $\Strg_*(A)\cong \K_*(\Locz(A\subset X))$ and $\K_*(\Roe(A))\cong \K_*(\Roe(A\subset X))$ which are induced by adjoining with a family of isometries or isometry, respectively, which cover the inclusion $\iota\colon A\to X$. 
Therefore, under these isomorphisms the induced maps $\iota_*\colon\Strg_*(A)\to\Strg_*(X)$ and $\K_*(\Roe(A))\to \K_*(\Roe(X))$ correspond exactly to the homomorphisms induced by the inclusions $\Locz(A\subset X)\subset\Locz(X)$ and $\Roe(A\subset X)\subset\Roe(X)$.

We immediately obtain the following diagram with exact rows and columns where all squares commute except of the ones bounded only by boundary maps, which commute up to the sign $-1$:
\begin{equation}\label{eq:RelativeHigsonRoeDiagram}
\xymatrix{
&\vdots\ar[d]&\vdots\ar[d]&\vdots\ar[d]&
\\\dots\ar[r]&\Strg_*(A)\ar[r]\ar[d]&\Strg_*(X)\ar[r]\ar[d]&\Strg_*(X,A)\ar[r]\ar[d]&\dots
\\\dots\ar[r]&\K_*(A)\ar[r]\ar[d]&\K_*(X)\ar[r]\ar[d]&\K_*(X,A)\ar[r]\ar[d]&\dots
\\\dots\ar[r]&\K_*(\Roe A)\ar[r]\ar[d]&\K_*(\Roe X)\ar[r]\ar[d]&\K_*(\Roe(X,A))\ar[r]\ar[d]&\dots
\\&\vdots&\vdots&\vdots&
}\end{equation}
Its last column is the relative version of the Higson--Roe sequence and it is dual to the first column of Diagram \eqref{eq:RelativeDualHigsonRoeDiagram}.

If $C\subset X$ is another closed subspaces, then we always have 
\begin{equation}\label{eq:Roesum}
\Roe(A\subset X)+\Roe(C\subset X)=\Roe(A\cup C\subset X)
\end{equation}
but the equation $\Roe(A\subset X)\cap\Roe(C\subset X)=\Roe(A\cap C\subset X)$ holds only under the so-called coarse excisiveness condition, i.\,e.\ if for all $R>0$ there is $S>0$ such that $B_R(A)\cap B_R(C)\subset B_S(A\cap C)$.
For our purposes it is important that the latter condition is always satisfied for the two subsets $X\times B,A\times Y$ of $X\times Y$, so we always have
\begin{equation}\label{eq:Roeproductintersection}
\Roe(X\times B\subset X\times Y)\cap\Roe(A\times Y\subset X\times Y)=\Roe(A\times B\subset X\times Y)\,.
\end{equation}
Similarly, the canonical analogues of \eqref{eq:Roesum} and \eqref{eq:Roeproductintersection} for the localization algebras also always hold true.

Now we can construct relative versions of the $*$-homomorphisms $\Phi,\Psi,\Upsilon_{\LSym},\bar{\Upsilon}_{\LSym}$ and $\bar{\Upsilon}_{\LSym,0}$.
The following Lemma tells us what these maps do with the ideals associated to the subspaces.
\begin{lem}\label{lem:PhiPsiUpsilonIdeals}
\begin{enumerate}
\item The $*$-homomorphism 
\[\Psi\colon\Roe(X \times Y) \tensmax \sHigCorRed Y \to  \Fipro(X)/\Roe(X)\,,\quad S\otimes [f]\mapsto [\tau(S)\circ\tilde\rho_Y(f)]\]
maps the ideal $\Roe(A\times Y\subset X\times Y)\tensmax\sHigCorRed Y$ of the domain into the ideal $(\Fipro(A\subset X)+\Roe(X))/\Roe(X)\cong\Fipro(A\subset X)/\Roe(A\subset X)$ of the target and it maps the ideal $\Roe(X\times B\subset X\times Y)\otimes\sHigCorRed(Y,B)$ to zero.

Therefore, it also maps $\Roe(X\times B\cup A\times Y\subset X\times Y)\tensmax \sHigCorRed(Y,B)$ into $(\Fipro(A\subset X)+\Roe(X))/\Roe(X)$. %

\item Similarily, 
the $*$-homomorphism
\begin{align*}
\bar{\Upsilon}_{\LSym} \colon \Loc(X \times Y) \tensmax \sHigComRed Y & \to \FiproLoc(X) / \Cz([1,\infty),\Roe(X)),\\
(T_t)_{t \in [1,\infty)} \otimes f & \mapsto \left[ (\tau(T_t) \circ \tilde\rho_Y(f))_{t \in [1,\infty)} \right].
\end{align*}
maps the ideal $\Loc(A\times Y\subset X\times Y)\tensmax \sHigComRed(Y)$ into the ideal 
\[\frac{\FiproLoc(A\subset X)+\Cz([1,\infty),\Roe(X))}{\Cz([1,\infty),\Roe(X))}\cong\frac{\FiproLoc(A\subset X)}{\Cz([1,\infty),\Roe(A\subset X)}\]
and the ideal $\Loc(X\times B\subset X\times Y)\tensmax \sHigComRed(Y,B)$ to zero.

Therefore, it also maps $\Loc(X\times B\cup A\times Y\subset X\times Y)\tensmax \sHigComRed(Y,B)$
into
\((\FiproLoc(A\subset X)+\Cz([1,\infty),\Roe(X)))/\Cz([1,\infty),\Roe(X))\).

\end{enumerate}
\end{lem}
\begin{proof}
If $S\in\Roe(X\times Y)$ is supported in an $R$-neighborhood of $A\times Y$ and $f\in\sHigComRed Y$, then $\tau(S)\circ\tilde\rho_Y(f)$ is clearly supported in an $R$-neighborhood of $A$. Furthermore, if $S$ is supported in an $R$-neighborhood of $X\times B$ and $[f]\in\sHigCorRed(Y,B)$, then the representative $f$ can be chosen such that it is zero on the $R$-neighborhood of $B$ and hence $\tau(S)\circ\tilde\rho_Y(f)=0$. This proves the first claim.

Now concerning the second claim, for $(T_t)_{t \in [1,\infty)}\in \Loc(A\times Y\subset X\times Y)$ and any $f\in\sHigComRed Y$ it is again clear that $(\tau(T_t) \circ \tilde\rho_Y(f))_{t \in [1,\infty)}\in\FiproLoc(A\subset X)$.
Next, let $(T_t)_{t \in [1,\infty)}\in \Loc(X\times B\subset X\times Y)$ one of the families of operators generating this \textCstar-algebra and $f\in\sHigComRed(Y,B)$. As $f$ is continuous and has vanishing variation, it is in particular uniformly continuous, so for each $\delta>0$ we can approximate $f$ by a function $g\in\sHigComRed(Y,B)$ which is zero on some $R$-neighborhood of $B$ and with $\|f-g\|<\delta$. Then for $t$ large enough the operator $T_t$ is supported within the $R$-neighborhood of $X\times B$ and therefore $\tau(T_t) \circ \tilde\rho_Y(g)=0$. This shows that $(\tau(T_t) \circ \tilde\rho_Y(f))_{t \in [1,\infty)}\in \Cz([1,\infty),\Roe(X))$ and the second claim follows.
\end{proof}

The lemma implies that we can divide out ideals from the $*$-homomorphisms $\Phi,\Psi,\Upsilon_{\LSym},\bar{\Upsilon}_{\LSym}$ and $\bar{\Upsilon}_{\LSym,0}$ to obtain their relative versions: We have
\begin{align*}
\Psi\colon\Roe((X,A)\times(Y,B))\tensmax \sHigCorRed (Y,B) &\to \frac{\Fipro(X)/\Roe(X)}{(\Fipro(A\subset X)+\Roe(X))/\Roe(X)}
\\&\cong \frac{\Fipro(X)/\Fipro(A\subset X)}{(\Fipro(A\subset X)+\Roe(X))/\Fipro(A\subset X)}
\\&= \Fipro(X,A)/\Roe(X,A)
\end{align*}
and similarily
\begin{alignat*}{2}
\Phi\colon\Roe((X,A)\times(Y,B))&\tensmax \sHigComRed(Y,B) &&\to \Fipro(X,A)/\Roe(X,A)
\\\bar{\Upsilon}_{\LSym} \colon \Loc((X,A)\times(Y,B)) &\tensmax \sHigComRed (Y,B) && \to \FiproLoc(X,A) / \Cz([1,\infty),\Roe(X,A))
\\\Upsilon_{\LSym} \colon \Loc((X,A)\times(Y,B)) &\otimes \Cz(Y\setminus B, \Kom) && \to \Loc(X,A) / \Cz([1,\infty),\Roe(X,A))
\\\bar{\Upsilon}_{\LSym,0} \colon \Locz((X,A)\times(Y,B)) &\tensmax \sHigComRed (Y,B) &&\to \FiproLocz(X,A) / \Cz((1,\infty),\Roe(X,A))\,.
\end{alignat*}

\begin{defn}\label{defn:relativeslantproducts}
We define the relative slant products between
the $\K$-theory of the Roe algebra and the $\K$-theory of the reduced stable Higson corona  as $(-1)^p$ times the composition
\begin{align*}
\K_p(\Roe((X,A)\times(Y,B))) &\otimes \K_{1-q}(\sHigCorRed (Y,B)) 
\\& \xrightarrow{\boxtimes} \ \K_{p+1-q}(\Roe((X,A)\times(Y,B))) \otimes_{\max} \sHigCorRed(Y,B))
\\& \xrightarrow{\Psi_*} \ \K_{p+1-q}(\Fipro(X,A)/\Roe(X,A))
\\& \xrightarrow{\partial} \ \K_{p-q}(\Roe (X,A))\,,
\end{align*}
between $\K$-homology and $\K$-theory as the composition
\begin{align*}
\K_p((X,A)\times(Y,B)) \otimes \K^{q}(Y,B)
&\cong \K_p(\Loc((X,A)\times(Y,B))) \otimes \K_{-q}(\Cz(Y\setminus B, \Kom))
\\&\xrightarrow{\boxtimes} \K_{p-q}(\Loc((X,A)\times(Y,B)) \otimes \Cz(Y\setminus B, \Kom))
\\&\xrightarrow{({\Upsilon_{\LSym}})_\ast} \K_{p-q}( \Loc(X,A) / \Cz([1, \infty), \Roe(X,A)))
\\&\xrightarrow[\cong]{(\ev_\infty)^{-1}} \K_{p-q}(\Loc(X,A)) \cong \K_{p-q}(X,A)
\end{align*}
and between the structure group and the $\K$-theory of the stable Higson compactification as the negative of the composition
\begin{align*}
\Strg_p((X,A)\times(Y,B))&\otimes\K_{-q}(\sHigComRed (Y,B))\cong
\\&\cong \K_p(\Locz((X,A)\times(Y,B))\otimes\K_{-q}(\sHigComRed (Y,B))
\\& \xrightarrow{\boxtimes} \K_{p-q}(\Locz((X,A)\times(Y,B)) \otimes_{\max} \sHigComRed (Y,B))
\\& \xrightarrow{(\bar{\Upsilon}_{\LSym,0})_*} \K_{p-q}(\FiproLocz(X,A) / \Cz((1,\infty),\Roe(X,A)))
\\& \xrightarrow{\partial} \K_{p-q-1}( \Cz((1,\infty),\Roe(X,A)))
\\&\xrightarrow[\cong]{\partial^{-1}}\K_{p-q}(\Roe(X,A))
\end{align*}
\end{defn}

The proofs of the results of the previous section and  \cite[Section 4]{EngelWulffZeidler} can directly be generalized to the relative case. In particular we have:
\begin{itemize}
\item The above slant product between relative $\K$-homology and relative $\K$-theory is exactly the same as the one obtained from the composition product in $E$-theory (cf.\ \cite[Lemma 4.19]{EngelWulffZeidler}).

\item The three relative slant products are independent of the choice of ample modules and they are natural under coarse, proper continuous or continuous coarse, respectively, maps between pairs of spaces (cf.\ \Cref{lem_LocSlantNatural} and \cite[Section 4.5]{EngelWulffZeidler}),

\item The slant products are compatible with each other under the maps of the relative Higson--Roe sequence and its dual (first column of \eqref{eq:RelativeHigsonRoeDiagram} and last column of \eqref{eq:RelativeDualHigsonRoeDiagram}, respectively), i.e.\ relative versions of \Cref{lem:structuregroupslantcompatiblewithRoestableHigsoncoronaslant,lem:structuregroupslantcompatiblewithhomologycohomologyslant} hold true.
\end{itemize}
We will give a more precise list of these properties in the next section, when we consider the even more general equivariant case.

A new aspect of relative slant products is that we can consider their compatibility with boundary maps in the long exact sequences associated to pairs of spaces.

\begin{lem}\label{lem:structuregroupslantcompatiblewithpairLESboundaryI}
The diagram
\[\xymatrix{
{\begin{array}{c}
\Strg_p((X,A)\times (Y,B))
\\\otimes\K_{-q}(\sHigComRed B)
\end{array}}
\ar[r]^{\id\otimes\partial}\ar[d]^{\partial\otimes\id}
&{\begin{array}{c}
\Strg_p((X,A)\times (Y,B))
\\\otimes\K_{-q-1}(\sHigComRed(Y,B))
\end{array}}\ar[dd]^{/}
\\{\begin{array}{c}
\Strg_{p-1}(X\times B\cup A\times Y,A\times Y)
\\\otimes\K_{-q}(\sHigComRed B)
\end{array}}\ar[d]_{\mathrm{exc}^{-1}}^{\cong}
&
\\{\begin{array}{c}
\Strg_{p-1}((X,A)\times B)
\\\otimes\K_{-q}(\sHigComRed B)
\end{array}}
\ar[r]^{/}
&\K_{p-q-1}(\Roe(X,A))
}\]
commutes up to a sign $(-1)^{p-1}$. 
\end{lem}
\begin{proof}

Let us introduce the abbreviations
\begin{align*}
C&\coloneqq \Locz(X\times Y)&I&\coloneqq\Locz(X\times B\cup A\times Y\subset X\times Y)\subset C
\\D&\coloneqq\sHigComRed Y&J&\coloneqq \sHigComRed(Y,B)\subset D
\end{align*}
such that 
\[C/I=\Locz((X,A)\times (Y,B))\qquad D/J=\sHigComRed(B)\]
and furthermore we define
\[E\coloneqq\FiproLocz(X,A) / \Cz((1,\infty),\Roe(X,A))\,.\]
\Cref{lem:PhiPsiUpsilonIdeals} implies that there is a map of short exact sequences
\[\xymatrix{
0\ar[r]&\frac{C\tensmax J+I\tensmax D}{I\tensmax J}\ar[r]\ar[d]^{\Upsilon'}&\frac{C\tensmax D}{I\tensmax J}\ar[r]\ar[d]&C/I\tensmax D/J\ar[r]\ar[d]&0
\\0\ar[r]&E\ar[r]&\ar[r]E&0\ar[r]&0
}\]
whose vertical maps are subquotients of the absolute version of $\bar\Upsilon_{\LSym,0}$ and we see that the composition 
\[\K_{*+1}\left(C/I\tensmax D/J\right)\xrightarrow{\partial}\K_{*}\left(\frac{C\tensmax J+I\tensmax D}{I\tensmax J}\right)\xrightarrow{(\Upsilon')_*} \K_{*}(E)\]
factors through $\K_*(0)=0$ and hence vanishes.
Note that there is a canonical direct sum decomposition 
\[\frac{C\tensmax J+I\tensmax D}{I\tensmax J}\cong (C/I\tensmax J)\oplus (I\tensmax D/J)\]
and let us denote by $i,j$ the inclusions of the two summands and by $r,s$ the projections onto the two summands.
The commutative diagram
\[\xymatrix{
0\ar[r]&C/I\tensmax J\ar[r]&C/I\tensmax D\ar[r]&C/I\tensmax D/J\ar[r]&0
\\0\ar[r]&\frac{C\tensmax J+I\tensmax D}{I\tensmax J}\ar[r]\ar[d]^{s}\ar[u]_{r}&\frac{C\tensmax D}{I\tensmax J}\ar[r]\ar[d]\ar[u]&C/I\tensmax D/J\ar[r]\ar@{=}[d]\ar@{=}[u]&0
\\0\ar[r]&I\tensmax D/J\ar[r]&C\tensmax D/J\ar[r]&C/I\tensmax D/J\ar[r]&0
}\]
shows that the three boundary maps associated to these short exact sequences, lets call them $\partial_{D/J},\partial,\partial_{C/I}$ to distinguish them, are related by $r_*\partial=\partial_{D/J}$ and $s_*\partial=\partial_{C/I}$. As $i_*+j_*\colon\K_*(C/I\tensmax J)\oplus \K_*(I\tensmax D/J)\to \K_*(\frac{C\tensmax J+I\tensmax D}{I\tensmax J})$ is the inverse to $(r_*,s_*)$, we have 
\[0=(\Upsilon')_*\partial=(\Upsilon')_*(i_*r_*+j_*s_*)\partial=(\Upsilon'\circ i)_*\partial_{D/J}+(\Upsilon'\circ j)_*\partial_{C/I}\,.\]

Note that
\begin{align*}
\Upsilon'\circ i\colon C/I\tensmax J&=\Locz((X,A)\times (Y,B))\tensmax\sHigComRed(Y,B)
\\&\to E=\FiproLocz(X,A) / \Cz((1,\infty),\Roe(X,A))
\end{align*}
is exactly the $*$-homomorphism $\bar\Upsilon_{\LSym,0}$ for the pair of spaces $(X,A)\times (Y,B)$.
In order to take care of the other summand, we observe that \Cref{lem:PhiPsiUpsilonIdeals} also implies that
\[\Upsilon'\circ j\colon I\tensmax D/J=\Locz(X\times B\cup A\times Y\subset X\times Y)\tensmax \sHigComRed(B)\to E\]
factors as $\Upsilon'\circ j=\Upsilon''\circ(t\otimes\id)$, where $t\colon I\to I/K$ is the quotient $*$-homomorphism dividing out the ideal $K\coloneqq \Locz(A\times Y\subset X\times Y)$.
Now, we may assume that the ample $X$-module $(H_X,\rho_X)$ contains an $\rho_X$-invariant subspace $H_A$ such that $(H_A,\rho_A\coloneqq\rho_X|_{\Cz(A)})$ is an ample $A$-module and analogously $(H_Y,\rho_Y)$ contains an ample $B$-module $(H_B,\rho_B\coloneqq\rho_Y|_{\Cz(B)})$. If we construct $\Locz(X\times B\cup A\times Y)$ as a subalgebra of $\Lin(H_A\otimes H_B\oplus H_A^\perp\otimes H_B\oplus H_A\otimes H_B^\perp)$ and $\Locz(X\times B)$ using the representation $\rho_X\otimes\rho_B$ as a subalgebra of $\Lin(H_X\otimes H_B)$, we obtain inclusions $\Locz((X,A)\times B)\subset \Locz(X\times B\cup A\times Y, A\times Y)\subset I/K$ which induce isomorphisms on $\K$-theory, the first of which is the excision isomorphism. It is clear that the restriction of $\Upsilon''\colon I/K\tensmax D/J\to E$ to $\Locz((X,A)\times B)\otimes \sHigComRed B$ is exactly the relative version of $\bar\Upsilon_{\LSym,0}$ for the pair of spaces $(X,A)\times B$.

Altogether we obtain the diagram
\[\xymatrix@C=3ex{
{\begin{array}{c}
\Strg_p((X,A)\times (Y,B))
\\\otimes\K_{-q}(\sHigComRed B)
\end{array}}
\ar[rr]^{\id\otimes\partial}\ar[dr]^-{\boxtimes}\ar[dd]^{\partial\otimes\id}
&&{\begin{array}{c}
\Strg_p((X,A)\times (Y,B))
\\\otimes\K_{-q-1}(\sHigComRed(Y,B))
\end{array}}\ar[d]^{\boxtimes}
\\&\K_{p-q}(C/I\tensmax D/J)\ar[r]^{\partial_{D/J}}\ar[d]^{t\circ\partial_{C/J}}
&\K_{p-q-1}(C/I\tensmax J)
\ar[d]_{(\bar\Upsilon_{\LSym,0})_*}^{=(\Upsilon'\circ i)_*}
\\{\begin{array}{c}
\Strg_{p-1}(X\times B\cup A\times Y,\\ A\times Y)
\\\otimes\K_{-q}(\sHigComRed B)
\end{array}}
\ar[r]^{\boxtimes}
&\K_{p-q-1}(I/K\tensmax D/J)\ar[r]^-{(\Upsilon'')_*}
&\K_{p-q-1}(E)
\\{\begin{array}{c}
\Strg_{p-1}((X,A)\times B)
\\\otimes\K_{-q}(\sHigComRed B)
\end{array}}
\ar[u]_-{\cong}\ar[r]^-{\boxtimes}
&{\begin{array}{c}
\K_{p-q-1}(\Locz((X,A)\times B)\\\tensmax \sHigComRed B)
\end{array}}
\ar[u]\ar[ur]_{(\bar\Upsilon_{\LSym,0})_*}
&
}\]
whose top-right quadrilateral commutes up to a sign $(-1)^p$, whose middle-right square commutes up to the sign $-1$ and whose other parts commute. The claim follows.
\end{proof}

\begin{lem}\label{lem:structuregroupslantcompatiblewithpairLESboundaryII}
The diagram
\[\xymatrix{
{\begin{array}{c}
\Strg_p((X,A)\times (Y,B))
\\\otimes\K_{-q}(\sHigComRed (Y,B))
\end{array}}
\ar[r]^{/}\ar[d]^{\partial\otimes\id}
&\K_{p-q}(\Roe(X,A))
\ar[dd]^{\partial}
\\{\begin{array}{c}
\Strg_{p-1}(X\times B\cup A\times Y,X\times B)
\\\otimes\K_{-q}(\sHigComRed (Y,B))
\end{array}}\ar[d]_{\mathrm{exc}^{-1}}^{\cong}
&
\\{\begin{array}{c}
\Strg_{p-1}(A\times (Y,B))
\\\otimes\K_{-q}(\sHigComRed (Y,B))
\end{array}}
\ar[r]^{/}
&\K_{p-q-1}(\Roe(A))
}\]
commutes.
\end{lem}

\begin{proof}
The claim follows from naturality of the boundary maps of $\K$-theory under the following map of short exact sequences:
\[\xymatrix{
0\ar[d]&0\ar[d]
\\\frac{\Locz(X\times B\cup A\times Y\subset X\times Y)}{\Locz(X\times B\subset X\times Y)}\tensmax \sHigComRed(Y,B)
\ar[d]\ar[r]%
&\frac{\FiproLocz(A\subset X)}{\Cz((1,\infty),\Roe(A\subset X))}
\ar[d]
\\\Locz(X\times Y,X\times B)\tensmax \sHigComRed(Y,B)
\ar[d]\ar[r]^-{\bar\Upsilon_{\LSym,0}}
&\frac{\FiproLocz(X)}{\Cz((1,\infty),\Roe(X))}\ar[d]
\\\Locz((X,A)\times (Y,B))\tensmax \sHigComRed(Y,B)
\ar[d]\ar[r]^-{\bar\Upsilon_{\LSym,0}}
&\frac{\FiproLocz(X,A)}{\Cz((1,\infty),\Roe(X,A))}\ar[d]
\\0&0
}\]
Under the inclusions of \textCstar-algebras mentioned at the end of the proof of the preceding lemma, the upmost horizontal arrow restricts to the relative version of $\bar\Upsilon_{\LSym,0}$ for the pair $A\times (Y,B)$, i.\,e.\ to
\[\bar\Upsilon_{\LSym,0}\colon\Locz(A\times(Y,B))\tensmax \sHigComRed(Y,B)\to\FiproLocz(A)/\Cz((1,\infty),\Roe(A))\,.\]
The boundary map associated to the left column is compatible with the boundary map associated to the triple $(X\times Y,X\times B\cup A\times Y, X\times B)$ under exterior products and the boundary map associated to the right column ``anticommutes'' with the last two boundary maps in the definition of the slant product, yielding a total sign of $(-1)^2=1$.
\end{proof}

\section{The equivariant case}

We now assume that $G,H$ are countable discrete groups acting properly and isometrically on $X$ and $Y$, respectively, and that the subspaces $A\subset X$ and $B\subset Y$ are invariant under these group actions. On $X\times Y$ we consider the product action of $G\times H$.

On the stable Higson corona and compactification, we implement the equivariance exactly as in \cite[Section 5.1]{EngelWulffZeidler} (see also \cite[Section 4.3]{wulff2020equivariant}): We define the $H$-equivariant $\K$-theory of $(Y,B)$ as
\[\K_H^*(Y,B)\coloneqq\K_{-*}(\Cz(Y\setminus B)\rtimes H)\,.\]
See \cite[p.977]{EngelWulffZeidler} for a discussion on why this is a good choice in the absolute case and hence clearly also in the relative case.
We did not specify which crossed product functor we used, because $\Cz(Y\setminus B) \rtimes H$ is independent of it. 

Furthermore, one always has $\Cz(Y\setminus B, \Kom) \rtimes H\cong(\Cz(Y\setminus B) \rtimes H)\otimes \Kom$.
Therefore, for any exact crossed product functor $\rtimes_\mu$ the short exact sequence
\[0\to \Cz(Y\setminus B, \Kom) \rtimes_{\mu} H \to \sHigComRed (Y,B) \rtimes_{\mu} H \to \sHigCorRed (Y,B) \rtimes_{\mu} H \to 0
\]
induces a long exact sequence
\begin{align*}
\dots&\to\K_{1-*}(\sHigCorRed (Y,B) \rtimes_{\mu} H)\xrightarrow{\mu_H^*} \K^*_H(Y,B)\to
\\&\qquad\to \K_{-*}(\sHigComRed (Y,B) \rtimes_{\mu} H)\to \K_{-*}(\sHigCorRed (Y,B) \rtimes_{\mu} H)\to\dots
\end{align*}
which is natural under continuous coarse maps between pairs of spaces.
Here, the connecting homomorphism $\mu_H^*$ is one version of a relative and equivariant coarse co-assembly map.

Recall that Emerson and Meyer originally constructed a slightly different equivariant coarse co-assembly map in \cite{EM_descent}, see also \cite[Section~2.3]{EM_coass}. It can also be adapted to the relative setup, where it is a map
\[\mu^\ast_\EM\colon \Ktop_*(H,\sHigCorRed (Y,B)) \to \K_H^{1-*}(Y,B)\,.\]
The two assembly maps are related by the commutative diagram
\[\xymatrix{
\Ktop_*(H,\sHigCorRed (Y,B)) \ar[rr]^-{\mu^\ast_\EM} \ar[dr]_-{\mu_\ast^\BC} && \K_H^{1-*}(Y,B)\\
& \K_\ast(\sHigCorRed (Y,B) \rtimes_{\mu} H) \ar[ur]_-{\mu^*_H} &
}\]
where $\mu^\BC_*$ is the Baum--Connes assembly map for $H$ with coefficients in $\sHigCorRed (Y,B)$.
Emerson and Meyer's co-assembly map is known to be an isomorphism or at least surjective in many important cases, implying that $\mu_H^*$ is also surjective. This observation is important, as it shows that there are non-trivial elements in $\K_\ast(\sHigCorRed (Y,B) \rtimes_{\mu} H)$ which can be used for slanting.

On the other side, the theory of equivariant Roe and localization algebras is well known, see, for example, \cite{WillettYuHigherIndexTheory}. We just need the relative version of it, which was also considered in \cite[Sections 4.4,6.4]{wulff2020equivariant}.

Let \(X\) be endowed with a proper isometric action of a countable discrete group \(G\) and take an \(X\)-module \((H_X, \rho_X)\).
In this case, we consider a unitary representation \(u_G \colon G \to \U(H_X)\) such that \((\rho_X, u_G)\) is a covariant pair, that is, $u_G(g) \rho_X(f) u_G(g)^\ast = \rho_X(g \cdot f)$ for all $g\in G$ and $f\in \Cz(X)$.
Then we say that \((H_X, \rho_X, u_G)\) is an \(X\)-\(G\)-module.
An \(X\)-\(G\)-module is said to be \emph{locally free} if for each finite subgroup \(F \subseteq G\) and any \(F\)-invariant Borel subset \(E \subseteq X\), there is a Hilbert space \(H_E\) such that \(1_E H_X\) and \(\elltwo(F) \otimes H_E\) are isomorphic as \(F\)-Hilbert spaces, where  \(\elltwo(F)\) is endowed with the left-regular representation and \(H_E\) is endowed with the trivial representation.
We say an \(X\)-\(G\)-module is \emph{ample} if it is ample as an \(X\)-module and locally free.
Ample \(X\)-\(G\)-modules always exist~\cite[Lemma~4.5.5]{WillettYuHigherIndexTheory}.

Now fix an ample \(X\)-\(G\)-module \((H_X, \rho_X, u_G)\).
Similarly as in the non-equivariant case, we get \textCstar\nobreakdash-algebras $\Roe[G] X$, $\Loc[G] X$ and $\Locz[G] X$ by taking the closure of \emph{equivariant} locally compact operators of finite propagation, respectively of suitable families $L$ of them.
They contain the respective ideals $\Roe[G](A\subset X)\coloneqq \Roe(A\subset X)\cap\Roe[G] X$,  $\Loc[G](A\subset X)\coloneqq \Loc(A\subset X)\cap\Loc[G] X$ and $\Locz[G](A\subset X)\coloneqq \Locz(A\subset X)\cap\Locz[G] X$ and dividing them out yields the relative versions $\Roe[G](X,A)\coloneqq\Roe[G] X/\Roe[G](A\subset X)$, $\Loc[G](X,A)\coloneqq\Loc[G] X/\Loc[G](A\subset X)$ and $\Locz[G](X,A)\coloneqq\Locz[G] X/\Locz[G](A\subset X)$.

As before, their \(\K\)-theory groups are independent of the choice of ample \(X\)-\(G\)-module, compare for instance \cite[Theorems~5.2.6, 6.5.7, Proposition~6.6.2]{WillettYuHigherIndexTheory}.
We also obtain a short sequence
\[0\to\Locz[G] (X,A)\to\Loc[G] (X,A)\xrightarrow{\operatorname{ev}_1}\Roe[G] (X,A)\to 0\]
and a corresponding long exact \(\K\)-theory sequence
\[\dots\to \K_{*+1}(\Roe[G] (X,A))\to \Strg_*^G(X,A)\to \K_*^G(X,A)\xrightarrow{\Ind} \K_*(\Roe[G] (X,A))\to\dots,\]
where we use \(\K_*^G(X,A) = \K_\ast(\Loc[G] (X,A))\) as our model for the relative \(\K\)-homology group.
If $G$ acts freely on $X$, then we have
\[\K_\ast(\Loc[G] (X,A)) \cong \K_*^G(X,A) \cong \K_\ast(G \backslash X,G\backslash A),\]
compare~\cite[Section 5.5]{EngelWulffZeidler} for the non-relative case.

Now we also fix an ample \(Y\)-\(H\)-module \((H_Y, \rho_Y, u_H)\).
We obtain another covariant pair \((\rho_Y^\prime, u_H\otimes\id_{\elltwo})\), where $\rho_Y^\prime\colon \Cz(Y, \Kom) \to \Lin(H_Y \otimes \elltwo)$ is the tensor product of $\rho_Y$ and the canonical representation of $\Kom$ on $\elltwo$. 
This, in turn, extends to a covariant pair \((\bar\rho_Y, u_H \otimes \id_{\elltwo})\) for the multiplier algebra \((\Mult(\Cz(Y, \Kom), H)\), where \(\bar\rho_Y\) is precisely the same as in \eqref{eq_rep_of_Cb}.
Finally, we construct a covariant pair \((\hat\rho_Y, \hat{u}_H)\), where \(\hat\rho_Y \coloneqq \id_{\elltwo(H)} \otimes \bar\rho_Y\), \(\hat{u}_H \coloneqq \lambda_H \otimes u_H \otimes \id_{\elltwo}\) and \(\lambda_H \colon H \to \U(\elltwo(H))\) is the left-regular representation.

Viewing \(\sHigComRed Y \subset \Lin(\elltwo(H) \otimes H_Y \otimes \elltwo)\) via the representation $\hat\rho_Y$, the above implies that the \(H\)-\Cstar-algebra \(\sHigComRed Y\) is covariantly represented on \( \elltwo(H) \otimes H_Y \otimes \elltwo\).
Then Fell's absorption principle \cite[Proposition~4.1.7]{BrownOzawaCstarFiniteDim} yields an embedding of the reduced crossed product \(\hat\rho_Y \rtimes_{\red} \hat{u}_H \colon \sHigComRed Y \rtimes_\red H \hookrightarrow \cB(\elltwo(H) \otimes H_Y \otimes \elltwo)\).

Now we will redefine \(H_{X \times Y} \coloneqq H_X \otimes \elltwo(H) \otimes H_Y\) endowed with the unitary representation \(u_{G \times H}\) of \(G \times H\) given by \(g,h \mapsto u_G(g) \otimes \lambda_H(h) \otimes u_H(h) \).
The Hilbert space \(H_{X \times Y}\) also supports a representation \(\rho_{X \times Y}\) of \(\Cz(X \times Y)\) via \(f \tens f^\prime \mapsto  \rho_{X}(f) \otimes \id_{\elltwo(H)} \otimes \rho_Y(f^\prime)\).
Together these turn \({H}_{X \times Y}\) into an ample \((X \times Y)\)-\((G \times H)\)-module.
As in \labelcref{eq_varRepOfCzX}, let \(\tilde{H}_X\coloneqq H_{X \times Y} \otimes \elltwo =H_X\otimes  \elltwo(H) \otimes H_Y\otimes \elltwo\) and define \(\tilde\rho_X\coloneqq \rho_X\otimes\id_{\elltwo(H) \otimes H_Y\otimes\elltwo}\colon \Cz(X)\to \Lin(\tilde H_X)\,\), \(\tilde\rho_Y \coloneqq \id_{H_X} \otimes \hat\rho_Y\) and \(\tilde{u}_{H} \coloneqq \id_{H_X} \otimes \hat{u}_H\).
Then \(\tilde\rho_Y \rtimes_{\red} \tilde{u}_H = \id_{H_X} \otimes (\hat\rho_Y \rtimes_{\red} \hat{u}_H) \colon \sHigComRed Y \rtimes_\red H \hookrightarrow \cB(\tilde H_X)\).
However, note that \(\tilde\rho_Y\) and \(\tilde\rho_X\) are defined in slightly different way than in \cref{sec:constr_noneq_slant} because of the additional tensor factor \(\elltwo(H)\).
In this way, we can use the reduced crossed product here.

Now, we also let \(\Fipro[G](\tilde\rho_X) \subset \cB(\tilde H_X)\) denote the \Cstar-algebra generated by all the \(G\)-equivariant operators of finite propagation. By redefining 
\[\tau\colon\Roe(\rho_{X\times Y})\to\Lin(\tilde H_X)\,,\quad S\mapsto S\otimes\id_{\ell^2}\,.\]
with these new representations and using 
\[\tilde\rho_Y \rtimes_{\mu} \tilde{u}_H \colon \sHigComRed Y \rtimes_{\mu} H \to \sHigComRed Y \rtimes_{\red} H \xrightarrow{\tilde\rho_Y \rtimes_{\red} \tilde{u}_H} \cB(\tilde H_X) \]
instead of $\tilde\rho_Y$, the obvious equivariant version of \Cref{lem:summarylemma} holds true.
Therefore, all constructions from the previous two sections can be performed in this modified equivariant setup.
We define all equivariant relative Roe and localization algebras in the canonical analogous fashion and obtain $*$-homomorphisms
\begin{alignat*}{1}
\Psi_\mu\colon\Roe[G\times H]((X,A)\times(Y,B))&\tensmax \sHigCorRed (Y,B)\rtimes_{\mu} H \to
\\&\to  \Fipro[G](X,A)/\Roe[G](X,A)
\\\Phi_\mu\colon\Roe[G\times H]((X,A)\times(Y,B))&\tensmax \sHigComRed(Y,B)\rtimes_{\mu} H  \to
\\&\to \Fipro[G](X,A)/\Roe[G](X,A)
\\\bar{\Upsilon}_{\LSym,\mu} \colon \Loc[G\times H]((X,A)\times(Y,B)) &\tensmax \sHigComRed (Y,B)\rtimes_{\mu} H  \to
\\& \to \FiproLoc[G](X,A) / \Cz([1,\infty),\Roe[G](X,A))
\\\Upsilon_{\LSym,\mu} \colon \Loc[G\times H]((X,A)\times(Y,B)) &\otimes \Cz(Y\setminus B, \Kom)\rtimes H  \to
\\& \to \Loc[G](X,A) / \Cz([1,\infty),\Roe[G](X,A))
\\\bar{\Upsilon}_{\LSym,\mu,0} \colon \Locz[G\times H]((X,A)\times(Y,B)) &\tensmax \sHigComRed (Y,B)\rtimes_{\mu} H  \to
\\&\to \FiproLocz[G](X,A) / \Cz((1,\infty),\Roe[G](X,A))
\end{alignat*}
which we use to define the following equivariant relative slant products.

\begin{defn}
Using the above modified $*$-homomorphisms, we define the equivariant relative slant products
\begin{alignat*}{2}
\K_p(\Roe[G\times H]((X,A)\times(Y,B))) &\otimes \K_{1-q}(\sHigCorRed (Y,B)\rtimes_{\mu} H) &&\to \K_{p-q}(\Roe[G] (X,A))\,,
\\\K^{G\times H}_p((X,A)\times(Y,B)) &\otimes \K^{q}_H(Y,B)&&\to \K^G_{p-q}(X,A)
\\\Strg^{G\times H}_p((X,A)\times(Y,B))&\otimes\K_{-q}(\sHigComRed (Y,B)\rtimes_{\mu} H)&&\to\K_{p-q}(\Roe[G](X,A))
\end{alignat*}
exactly as in \Cref{defn:relativeslantproducts}.
\end{defn}

The proofs of all properties of the non-equivariant slant products presented in the previous two sections also go through in this modified equivariant setup. Let us summarize these properties.

\begin{thm}\label{thm:slantproductpropertiessummarized}
The equivariant relative slant product between the structure group and the $\K$-theory of the stable Higson compactification has the following properties:
\begin{enumerate}
\item It is up to canonical isomorphism independent of the choice of ample modules 
 and natural under pairs of equivariant continuous coarse maps $\alpha\colon (X,A)\to (X',A')$ and $\beta\colon (Y,B)\to (Y',B')$ in the sense that
\begin{equation*}
\alpha_*(x/\beta^*(\theta))=(\alpha\times \beta)_*(x)/\theta
\end{equation*}
for all $x\in \Strg^{G\times H}_*((X,A)\times (Y,B))$ and $\theta\in \K^H_*(\sHigComRed (Y',B'))$.

\item It is compatible with the other slant products under the maps of the Higson--Roe exact sequence and its dual: 
The diagram
\[\xymatrix{
\K^{G\times H}_p((X,A)\times (Y,B))\otimes\K^q(Y,B)
\ar[r]^-{/}
&\K^G_{p-q}(X)
\ar[dd]^-{\Ind}
\\\Strg^{G\times H}_p((X,A)\times (Y,B))\otimes\K^q(Y,B)
\ar[u]\ar[d]&
\\\Strg^{G\times H}_p((X,A)\times (Y,B))\otimes\K_{-q}(\sHigComRed  (Y,B)\rtimes_\mu H)
\ar[r]^-{/}
&\K_{p-q}(\Roe[G](X))
}\]
commutes and the diagram
\[\xymatrix{
{\begin{array}{c}\K_p(\Roe[G\times H]((X,A)\times (Y,B))\\\otimes\K_{1-q}(\sHigComRed (Y,B)\rtimes_\mu H)\end{array}}
\ar[r]^-{\partial\otimes\id}\ar[d]
&{\begin{array}{c}\Strg^{G\times H}_{p-1}((X,A)\times (Y,B))\\\otimes\K_{1-q}(\sHigComRed (Y,B)\rtimes_\mu H)\end{array}}
\ar[d]^-{/}
\\{\begin{array}{c}\K_p(\Roe[G\times H]((X,A)\times (Y,B)))\\\otimes\K_{1-q}(\sHigCorRed (Y,B)\rtimes_\mu H)\end{array}}
\ar[r]^-{/}
&\K_{p-q}(\Roe[G](X))
}\]
commutes up to the sign $(-1)^p$.

\item It is compatible with the obvious equivariant relative generalization of the external product: If $G,H,K$ are countable discrete groups acting properly and isometrically on $(X,A),(Y,B),(Z,C)$ and if $Z$ has continuously bounded geometry, then the diagram
\[\xymatrix{
{\begin{array}{c}\K^G_m(X,A)\otimes\Strg^{H\times K}_p((Y,B)\times (Z,C))\\\otimes\K_{-q}(\sHigComRed (Z,C)\rtimes_\mu K)\end{array}}
\ar[r]^-{\id\otimes /}\ar[d]^{\times\otimes\id}
&\K_m^G(X)\otimes\K_{p-q}(\Roe[H](Y,B))
\ar[d]^{\times\circ(\Ind\otimes \id)}
\\{\begin{array}{c}\Strg^{G\times H\times K}_{m+p}((X,A)\times (Y,B)\times (Z,C))\\\otimes \K_{-q}(\sHigComRed (Z,C)\rtimes_\mu K)\end{array}}
\ar[r]^-{/}
&\K_{m+p-q}(\Roe[G\times H](X\times Y))
}\]
commutes.

\item It is compatible with the connecting homomorphisms in the long exact sequences associated to pairs of spaces:
The diagram
\[\xymatrix{
{\begin{array}{c}
\Strg^{G\times H}_p((X,A)\times (Y,B))
\\\otimes\K_{-q}(\sHigComRed (Y,B)\rtimes_\mu H)
\end{array}}
\ar[r]^-{/}\ar[d]^{\partial\otimes\id}
&\K_{p-q}(\Roe[G](X,A))
\ar[dd]^{\partial}
\\{\begin{array}{c}
\Strg^{G\times H}_{p-1}(X\times B\cup A\times Y,X\times B)
\\\otimes\K_{-q}(\sHigComRed (Y,B)\rtimes_\mu H)
\end{array}}\ar[d]_{\mathrm{exc}^{-1}}^{\cong}
&
\\{\begin{array}{c}
\Strg^{G\times H}_{p-1}(A\times (Y,B))
\\\otimes\K_{-q}(\sHigComRed (Y,B)\rtimes_\mu H)
\end{array}}
\ar[r]^-{/}
&\K_{p-q-1}(\Roe[G](A))
}\]
commutes and the diagram
\[\xymatrix{
{\begin{array}{c}
\Strg^{G\times H}_p((X,A)\times (Y,B))
\\\otimes\K_{-q}(\sHigComRed B\rtimes_\mu H)
\end{array}}
\ar[r]^-{\id\otimes\partial}\ar[d]^{\partial\otimes\id}
&{\begin{array}{c}
\Strg^{G\times H}_p((X,A)\times (Y,B))
\\\otimes\K_{-q-1}(\sHigComRed(Y,B)\rtimes_\mu H)
\end{array}}\ar[dd]^{/}
\\{\begin{array}{c}
\Strg^{G\times H}_{p-1}(X\times B\cup A\times Y,A\times Y)
\\\otimes\K_{-q}(\sHigComRed B\rtimes_\mu H)
\end{array}}\ar[d]_{\mathrm{exc}^{-1}}^{\cong}
&
\\{\begin{array}{c}
\Strg^{G\times H}_{p-1}((X,A)\times B)
\\\otimes\K_{-q}(\sHigComRed B\rtimes_\mu H)
\end{array}}
\ar[r]^{/}
&\K_{p-q-1}(\Roe[G](X,A))
}\]
commutes up to a sign $(-1)^{p-1}$.

\end{enumerate}
\end{thm}

\section{Classifying spaces and \texorpdfstring{$\sigma$}{\textsigma}-spaces}

In the next section, we are going to apply our slant products to the classifying spaces $\Efree G$ for free $G$-actions and $\Eub G$ for proper $G$-actions. As these spaces are in general not locally compact and whence cannot be equipped with proper metrics, we have to explain how this is done.

Any $G$-invariant, $G$-compact subset $X$ of $\Efree G$ or $\Eub G$ can be equipped with a canonical quasi-isometry class of $G$-invariant proper metrics. Hence we can simply define the homological groups $\K^G_*(\Efree G)$, $\Strg^G_*(\Efree G)$, $\K_*(\Roe[G](\Efree G))$ as well as $\K^G_*(\Eub G)$, $\Strg^G_*(\Eub G)$, $\K_*(\Roe[G](\Eub G))$ by taking the direct limits of the groups $\K^G_*(X)$, $\Strg^G_*(X)$, $\K_*(\Roe[G](X))$ over all subspaces $X$ of this type.
Note that we always have $\K_*(\Roe[G](\Efree G))\cong\K_*(\CstarRed G)\cong \K_*(\Roe[G](\Eub G))$ canonically.
Moreover, we will see in \Cref{lem:autosigma,lem:GCWmodelsforEG} below that one can always choose $G$-CW-models for $\Efree G$ and $\Eub G$ which contain cofinal sequences $X_0\subset X_1\subset\dots$ of such subspaces, over which the direct limits can be taken. 

The cohomological groups, however, should not be defined in the dual way by taking inverse limits, but instead one would expect them to satisfy $\varprojlim^1$-sequences
\[0\to {\varprojlim_{n\in\N}}^1\,\K_H^{*-1}(X_n)\to\K_H^*(\Efree G) \to\varprojlim_{n\in\N}\K_H^*(X_n)\to 0\]
and analogously for $\K_*(\sHigComRed(\Efree G)\rtimes_\mu H)$ and $\K_*(\sHigCorRed(\Efree G)\rtimes_\mu H)$ as well as the corresponding groups for $\Eub G$.
The way to go is to consider $\Efree G$ and $\Eub G$ as $\sigma$-spaces, i.\,e.\ spaces filtered by an increasing sequence of subspaces, following an idea pioneered by Emerson and Meyer in \cite[Sec.~2]{EmeMey}, see also \cite[Sec.~3]{WulffCoassemblyRinghomo} and \cite[Definitions~5.1--5.3]{wulff2020equivariant}. 
We make the following definition, which specifies the type of $\sigma$-space to which the theory of the preceding sections can be applied swiftly.

\begin{defn}
\begin{itemize}
\item A \emph{$\sigma$-$G$-metric space}  $\cX$ is an increasing sequence $X_0\subset X_1\subset X_2\subset\dots$ of proper metric spaces of continuously bounded geometry equipped with proper isometric $G$-actions such that for all $m\leq n$ the inclusion $X_m\subset X_n$ is a $G$-equivariant continuous coarse embedding with closed image.

By an abuse of notation we will also write $\cX$ for the set $\cX=\bigcup_{n\in\N}X_n$ equipped with the final topology. A subset $\cA\subset \cX$ is open/closed if and only if each intersection $A_n\coloneqq\cA\cap X_n$ is open/closed.
\item We say that another $\sigma$-$G$-metric space $\cA=\bigcup_{n\in\N}A_n$ is a \emph{subspace} of $\cX$ if $\cA\subset \cX$, $A_n=\cA\cap X_n$ and each $A_n\subset X_n$ carries the restricted metric and $G$-action. Note that these correspond exactly to the closed $G$-invariant subsets of $\cX$.
\item We call a map $f\colon \cX\to \cY$ between $\sigma$-$G$-metric spaces a \emph{$\sigma$-map} if \linebreak 
$\forall m\exists n\colon f(X_m)\subset Y_n$. Note that $f$ is continuous with respect to the final topologies if and only if all the restricted maps $f|_{X_m}\colon X_m\to Y_n$ are continuous.
We call $f$ \emph{$\sigma$-coarse} if all the restrictions $f|_{X_m}\colon X_m\to Y_n$ are coarse maps.

\item We say that a \(\sigma\)-\(G\)-metric space $\cX=\bigcup_{n\in\N}X_n$ is \emph{\(\sigma\)-\(G\)-compact} if each \(X_n\) is \(G\)-compact.
\end{itemize}
\end{defn}

The reader should stay alert that the definition asks for three properties, namely properness of the metric, continuously bounded geometry and properness as well as isometry of the $G$-actions, which are not reflected in the name \enquote{$\sigma$-$G$-metric space}. We decided to use this shorter, less descriptive name and also not consider more general types of $\sigma$-spaces solely out of convenience, because any more general setup is not relevant for our applications in the final section.

Proper metric spaces $X$ of continuously bounded geometry equipped with an isometric $G$-action can be considered as $\sigma$-$G$-metric spaces by equipping them with the trivial filtration $X=X_0=X_1=\dots$.

Other important examples of such spaces are proper regular $G$-CW-complexes which are filtrated by a sequence of increasing $G$-invariant, $G$-finite subcomplexes.
By properness we mean that all the cell stabilizers are finite, which implies that all $G$-finite subcomplexes are locally compact and the restricted $G$-action on them is proper.
It remains to metrize the subcomplexes and, although all $G$-invariant metrics on them are proper and coarsely equivalent to each other, we need to choose ones which have continuously bounded geometry. To facilitate this, we assume that the $G$-CW-complex is regular, that is, that all closed cells $G/H\times D^k$ are embedded homeomorphically into the complex. Indeed, it is readily deduced by induction from the following proposition that all proper regular $G$-finite $G$-CW-complexes and even pairs of them can be metrized with continuously bounded geometry.

\begin{prop}\label{prop:cellgluingcontinuouslyboundedgeometry}
Let $Y$ be a proper metric space of continuously bounded geometry equipped with a proper isometric $G$-action and $f\colon G/H\times S^{k-1}\to Y$ an injective $G$-equivariant continuous map with $H\subset G$ a finite subgroup.
Then the proper $G$-space $Y\cup_f G/H\times D^k$ can also be equipped with a proper equivariant metric of continuously bounded geometry. 
\end{prop}

First we need to prove a lemma.

\begin{lem}
Let $Y$ be a compact metric space of continuously bounded geometry. Then the cone $X\coloneqq Y\times[0,1]/Y\times\{0\}$ over $Y$ can be equipped with a metric of continuously bonded geometry whose restriction to $Y\times \{1\}$ is the original metric on $Y$.
\end{lem}

\begin{proof}
Pick any isometric embedding $\iota\colon Y\to V$ into a normed space $(V,\|\blank\|)$ (e.\,g.\ the Kuratowski embedding) and let $C\coloneqq\sup_{y\in Y}\|y\|$, which is finite because $Y$ is compact. Then the cone $X$ can be embedded into the normed space $V\oplus_1 \R$ via $(y,t)\mapsto (ty,Ct)$ and we shall equip it with the metric induced by this norm, that is,
\[d_X((x,s),(y,t))=C|s-t|+\|sx-ty\|\,.\] 

Let $\hat Y_r\subset Y$ be the subsets witnessing continuous bounded geometry with corresponding bounds $K_{r,R}$. 
For each $n\in\N\setminus\{0\}$ we consider the subsets
\[\hat X_{\frac{C}{n}}\coloneqq \bigcup_{i=0}^n \hat Y_{\frac{C}{2i}}\times \left\{\frac{i}{n}\right\}\subset X\]
where we let $\hat Y_{C/0}$ denote any nonempty subset, such that we obtain a single point at the apex of the cone.

The $\frac{C}{n}$-balls with centers in $\hat X_{\frac{C}{n}}$ cover $X$: Given $(x,s)\in X$ there exist $i\in\{0,\dots,n\}$ with $|S-\frac{i}{n}|\leq \frac{1}{2n}$ and $y\in \hat Y_{\frac{C}{2i}}$ with $d_Y(x,y)<\frac{C}{2i}$, so that
\begin{align*}
d_X\left((x,s),\left(y,\frac{i}{n}\right)\right)&\leq d_X\left((x,s),\left(x,\frac{i}{n}\right)\right)+d_X\left(\left(x,\frac{i}{n}\right),\left(y,\frac{i}{n}\right)\right)
\\&=C\left|s-\frac{i}{n}\right|+\frac{i}{n}d_Y(x,y)<\frac{C}{n}\,.
\end{align*}

On the other hand, for all $(x,s)\in X$ and $(y,\frac{i}{n})\in\hat X_{\frac{C}{n}}$  we have 
\begin{align*}
d_X\left((x,s),\left(y,\frac{i}{n}\right)\right)&= C\left|s-\frac{i}{n}\right|+\left\|sx-\frac{i}{n}y\right\|
\\&\geq C\left|s-\frac{i}{n}\right|+\frac{i}{n}\|x-y\|-\left|s-\frac{i}{n}\right|\|x\|
\\&\geq \frac{i}{n}d_Y(x,y)
\end{align*}
where the last inequality is due to our choice of the constant $C$. This implies that for any $\alpha>0$ the number of elements of
\[\hat X_{\frac{C}{n}}\cap \clBall_{\alpha\cdot\frac{C}{n}}(x,s)\subset \bigcup_{|i-sn|\leq \alpha} \left(\hat Y_{\frac{C}{2i}}\cap\clBall_{\alpha\frac{C}{i}} (x) \right)\times \left\{\frac{i}{n}\right\}\]
is bounded by 
\[L_{\frac{C}{n},\alpha\frac{C}{n}}\coloneqq (2\lceil\alpha\rceil+1)\cdot \sup_{i\in\N\setminus\{0\}}K_{\frac{C}{2i},2\alpha\frac{C}{2i}}\]
and this is finite, because the supremum is taken over a sequence with finite limit superior. 

For arbitrary $r>0$ choose $n\in\N\setminus\{0\}$ with $\frac{C}{n}\leq r<\frac{C}{n-1}$ and define $\hat X_r\coloneqq \hat{X}_{\frac{C}{n}}$.
It is then clear that these sets witness continuously bounded geometry for suitably chosen constants $L_{r,R}$.
\end{proof}

\begin{proof}[Proof of \Cref{prop:cellgluingcontinuouslyboundedgeometry}]
Since $Y$ has continuously bounded geometry, so do the compact subspaces $Z^g\coloneqq f(\{gH\}\times S^{k-1})$. Indeed, this is readily checked by choosing subsets $\hat Z^g_r\subset Z^g$ that are sufficiently close to the subsets $\hat Y_r$ witnessing the continuously bounded geometry of $Y$. Then each cell $\{gH\}\times D^k$ can be identitfied with the cone over $Z^g$, because $f$ is injective, and hence can be equipped with a metric of continuously bounded geometry by the preceding lemma such that $f$ restricts to an isometry on the boundary $\{gH\}\times S^{k-1}$.

We do this for all cells $\{gH\}\times D^k$ in the same way and thereby construct a $G$-equivariant proper metric on $Y\cup_f G/H\times D^k$. It is a locally finite union of metric spaces of continuously bounded geometry and the bounds $K_{r,R}$ can be chosen equal for all of them. By taking the union of the witnessing subsets it is now evident that $Y\cup_f G/H\times D^k$ must also have continuously bounded geometry.
\end{proof}

\begin{lem}\label{lem:autosigma}
Let $\cX=\bigcup_{m\in\N}X_m$ be a \(\sigma\)-\(G\)-compact $\sigma$-$G$-metric space and let $\cY=\bigcup_{n\in\N}Y_n$ be $G$-CW-complexes filtrated as above. Then any $G$-equivariant continuous map $f\colon \cX\to \cY$ is automatically also a $\sigma$-coarse $\sigma$-map.
\end{lem}

\begin{proof}
As each $X_m$ is $G$-compact, we have $X_m=G\cdot K$ for a compact subset $K\subset X_m$. Then $f(K)$ is compact and thus contained in a finite subcomplex of $\cY$. Therefore, there exists $n\in\N$ with $f(K)\subset Y_n$, so $f(X_m)=G\cdot f(K)\subset Y_n$. This shows that $f$ is a $\sigma$-map.
Coarseness of the restrictions follows from $G$-equivariance, because all $X_m,Y_n$ are $G$-compact.
\end{proof}

The following lemma shows that the classifying spaces indeed fit into this set-up and that their homotopy type is still unique, despite putting the additional structure of a filtration on them.

\begin{lem}\label{lem:GCWmodelsforEG}
There are $G$-CW-models for $\Efree G$ and $\Eub G$ with countably many cells and which, therefore, are the union of an increasing sequence of $G$-invariant, $G$-finite subcomplexes.
Any two choices of such models equipped with any sequences of subcomplexes of this type are homotopy equivalent in the category of $\sigma$-$G$-metric spaces and $G$-equivariant continuous $\sigma$-coarse maps. 
\end{lem}

\begin{proof}
There are well-known constructions for the classifying spaces as $\Delta$-complexes:
The model $\Efree_\Delta G$ for $\Efree G$ has one $n$-simplex for each ordered $(n+1)$-tuple of elements of $G$ and the $\Eub_\Delta G$ model for $\Eub G$ has one $n$-simplex for each $(n+1)$-element subset of $G$. These simplices are glued together in the obvious way, so in particular $\Eub_\Delta G$ is the geometric realization of the full simplicial complex with vertex set $G$. 
These complexes consist of countably many simplices, because $G$ was assumed to be countable.

The complexes $\Efree_\Delta G$ and $\Eub_\Delta G$ can, for example, be filtered by the sequence of their skeleta  $\Efree_\Delta^n G$ and $\Eub_\Delta^n G$, respectively. Another reasonable choice for the filtration of $\Eub_\Delta G$ is by the Rips complexes $P_n(G)$, which consist of all the simplices of diameter at most $n$.

Now assume that $\cX$ and $\cY$ are two $G$-CW-models for either $\Efree G$ or $\Eub G$ with countably many cells.
They are known to be homotopy equivalent as topological $G$-spaces. 
Let $f\colon\cX\to\cY$ be any $G$-equivariant homotopy equivalence with $G$-equivariant homotopy inverse $g\colon \cY\to \cX$ and $G$-equivariant homotopies $h\colon \cX\times[0,1]\to\cX$, $k\colon\cY\times[0,1]\to\cY$ between $g\circ f$, $f\circ g$ and the respective identities. 
Then all of $f,g,h,k$ are not only continuous, but also $\sigma$-coarse $\sigma$-maps by the previous lemma.
\end{proof}

Now we can define the groups associated to pairs of such spaces, starting with the homological ones.
\begin{defn}\label{defn:HigsonRoeRelativeEquivariantSigma}
Let $\cX=\bigcup_{n\in\N}X_n$ be a $\sigma$-$G$-metric space with closed subspace $\cA=\bigcup_{n\in\N}A_n$. Then we define the groups
\begin{align*}
\K_*^G(\cX,\cA)&\coloneqq \varinjlim_{n\in\N}\K^G_*(X_n,A_n)
\\\Strg_*^G(\cX,\cA)&\coloneqq \varinjlim_{n\in\N}\Strg_*^G(X_n,A_n)
\\\K_*(\Roe[G] (\cX,\cA))&\coloneqq \varinjlim_{n\in\N}\K_*(\Roe[G] (X_n,A_n))
\end{align*}
which canonically fit into a long exact sequence
\begin{equation*}\label{eq:HigsonRoeRelativeEquivariantSigma}
\dots\to \K_{*+1}(\Roe[G] (\cX,\cA))\to \Strg_*^G(\cX,\cA)\to \K_*^G(\cX,\cA)\xrightarrow{\Ind} \K_*(\Roe[G] (\cX,\cA))\to\dots
\end{equation*}
called the \emph{relative equivariant Higson-Roe sequence for $\sigma$-spaces}.
\end{defn}

The dual groups for pairs $(\cY,\cB)$ of $\sigma$-$H$-metric spaces will be defined by evoking $\K$-theory for $\sigma$-\textCstar-algebras (cf.~\cite{PhiRep,phillips}) as described below. We do not summarize the relevant properties of $\K$-theory for $\sigma$-\textCstar-algebras here, because they are exactly the same as in the case of \textCstar-algebras.

Recall that a $\sigma$-\textCstar-algebra is a topological $*$-algebra $C$ whose topology is Hausdorff, is generated by an increasing sequence of \textCstar-seminorms $p_0\leq p_1\leq p_2\leq \dots$ and is complete with respect to these seminorms. Equivalently, it is an inverse limit $C=\varprojlim_{n\in\N}C_n$ of a sequence of \textCstar-algebras $\dots\to C_2\to C_1\to C_0$ (take $C_n\coloneqq C/\ker(p_n)$). If $C$ is even a $\sigma$-$H$-\textCstar-algebra, that is, the inverse limit of a sequence of $H$-\textCstar-algebras $C_n$, then one can define the $\sigma$-\textCstar-algebra $C\rtimes_\mu H\coloneqq\varprojlim_{n\in\N}C_n\rtimes_\mu H$.

\begin{defn}
Given a $\sigma$-$H$-metric space $\cY=\bigcup_{n\in\N}Y_n$ with a subspace $\cB=\bigcup_{n\in\N}B_n$,
we define the $\sigma$-$H$-\textCstar-algebras
\begin{align*}
\Cz(\cY\setminus\cB)&\coloneqq\varprojlim_{n\in\N}\Cz(Y_n\setminus B_n)
\\\sHigComRed(\cY,\cB)&\coloneqq\varprojlim_{n\in\N}\sHigComRed(Y_n,B_n)
\\\sHigCorRed(\cY,\cB)&\coloneqq\varprojlim_{n\in\N}\sHigCorRed(Y_n,B_n)
\end{align*}
and the $\K$-theory group $\K_H^*(\cY,\cB)\coloneqq\K_{-*}(\Cz(\cY\setminus\cB)\rtimes_\mu H)$.
\end{defn}
The meaning of the other two cohomological groups $\K_*(\sHigComRed(\cY,\cB)\rtimes_\mu H)$ and  $\K_*(\sHigCorRed(\cY,\cB)\rtimes_\mu H)$ is now clear, too.

Note that we have a short exact sequence
\[0\to\Cz(\cY,\cB)\rtimes_\mu H\to\sHigComRed(\cY,\cB)\rtimes_\mu H\to \sHigCorRed(\cY,\cB)\rtimes_\mu H\to 0\]
inducing a long exact sequence in $\K$-theory and similarily we get long exact pair-sequences for $\K_H^*(\blank,\blank)$, $\K_*(\sHigComRed(\blank,\blank)\rtimes_\mu H)$ and $\K_*(\sHigCorRed(\blank,\blank)\rtimes_\mu H)$.
Furthermore, we indeed have a $\varprojlim^1$-sequence
\[0\to {\varprojlim_{n\in\N}}^1\,\K_H^{*-1}(Y_n,B_n)\to\K_H^*(\cY,\cB) \to\varprojlim_{n\in\N}\K_H^*(Y_n,B_n)\to 0\]
and analogously for $\K_*(\sHigComRed(\cY,\cB)\rtimes_\mu H)$ and $\K_*(\sHigCorRed(\cY,\cB)\rtimes_\mu H)$.

\Cref{lem:GCWmodelsforEG} implies that the homotopy types of the $\sigma$-$H$-\textCstar-algebras $\Cz(\Eub G)$ and $\Cz(\Efree G)$ are independent of the choice of models for $\Efree G,\Eub G$ and their filtration. Therefore, their $\K$-theories are, up to canonical isomorphism, also independent of it. Moreover, $\sHigCorRed(\Efree G)\cong\sHigCorRed(G)\cong\sHigCorRed(\Eub G)$ are in fact \textCstar-algebras, because they are inverse limits over a system of $*$-isomorphisms.

Now, due to the naturality stated in \Cref{thm:slantproductpropertiessummarized}, the slant product immediately passes over to the limits and colimits and we get a slant product 
\begin{align*}
\Strg^{G\times H}_p&((\cX,\cA)\times(\cY,\cB))\otimes\K_{-q}(\sHigComRed (\cY,\cB)\rtimes_{\mu} H)\to
\\&\to\varinjlim_{n\in\N}\Strg^{G\times H}_p((X_n,A_n)\times(Y_n,B_n))\otimes\varprojlim_{n\in\N}\K_{-q}(\sHigComRed (Y_n,B_n)\rtimes_{\mu} H)\to
\\&\to\varinjlim_{n\in\N}\K_{p-q}(\Roe[G](X_n,A_n))
=\K_{p-q}(\Roe[G](\cX,\cA))
\end{align*}
where $ (\cX,\cA)\times(\cY,\cB)$ denotes, of course, the pair of $\sigma$-$(G\times H)$-metric spaces given by the filtration $(X_n\times Y_n,X_n\times B_n\cup A_n\times Y_n)$.
Completely analogously, one also obtains the $\sigma$-versions
\begin{alignat*}{2}
\K_p(\Roe[G\times H]((\cX,\cA)\times(\cY,\cB))) &\otimes \K_{1-q}(\sHigCorRed (\cY,\cB)\rtimes_{\mu} H) &&\to \K_{p-q}(\Roe[G] (\cX,\cA))\,,
\\\K^{G\times H}_p((\cX,\cA)\times(\cY,\cB)) &\otimes \K^{q}_H(\cY,\cB)&&\to \K^G_{p-q}(\cX,\cA)
\end{alignat*}
of the other two slant products.
The properties of the slant products summarized in \Cref{thm:slantproductpropertiessummarized} readily pass over to these $\sigma$-versions.

\section{Numerical invariants on the analytic structure group}
\label{sec:numerical}
\begin{defn}
The pairing
\[\langle\blank,\blank\rangle\colon \Strg^H_p(\cY,\cB) \otimes \K_{-q}(\sHigComRed (\cY,\cB)\rtimes_\mu H)\to \K_{p-q}(\Roe\{*\})\]
is defined as the special case of the equivariant slant product where $(\cX,\cA)=(\{*\},\emptyset)$ is a single point equipped with the action of the trivial group.
\end{defn}
In particular, given a $\sigma$-$G$-metric space \(\cX\), we obtain a numerical pairing involving the analytic structure group and the stable Higson corona of the following type,
\begin{equation}
  \langle\blank,\blank\rangle\colon \Strg^G_n(\cX) \otimes \K_{-n}(\sHigComRed (\cX)\rtimes_\mu G)\to \Z. \label{eq:numerical_pairing_structure_group}
\end{equation}

Given a \(\sigma\)-\(G\)-compact \(\sigma\)-\(G\)-metric space $\cX$, we always obtain a pair of the form $(\Eub G,\cX)$. 
This has to be understood as follows. Any classifying map $f\colon \cX\to\Eub G$ of the proper action into a $CW$-model for $\Eub G$ with countably many cells as in \Cref{lem:GCWmodelsforEG} is a continuous $G$-equivariant $\sigma$-coarse $\sigma$-map by \Cref{lem:autosigma}. By redefining $\Eub G$ as the mapping cylinder $\Eub G\coloneqq\operatorname{Zyl}(f)$, which is now a $\sigma$-$G$-metric space, we retrieve $\cX$ as a subspace of $\Eub G$.

Note also that since classifying maps are unique up to equivariant homotopy, the same argument shows that the homotopy is also a $\sigma$-map, but it might not be a homotopy of inclusions.

\begin{lem} \label{lem:canonical_relative_iso_dual}
  Let \(\cX\) be \(\sigma\)-\(G\)-compact.
  Then the canonical map
  \[\K^{-*}_G(\Eub G,\cX)\to \K_*(\sHigComRed(\Eub G,\cX)\rtimes_\mu G)\]
  is an isomorphism.
\end{lem}
\begin{proof}
Let $\cX=\bigcup_{n\in\N}X_n$ and $\Eub G=\bigcup_{n\in\N}Y_n$ be the filtrations of these $\sigma$-spaces. As explained above, we may see $\cX$ as a subspace of $\Eub G$ and then $X_n=\cX\cap Y_n$. Since all $X_n,Y_n$ are $G$-compact, the inclusions $X_n\subset Y_n$ are coarse equivalences. Therefore, $\sHigCorRed(Y_n,X_n)=0$ and $\sHigComRed(Y_n,X_n)=\Cz(Y_n\setminus X_n)$ for each $n\in\N$, and thus \(\sHigCorRed(\Eub G, \cX) = 0\) and \(\sHigComRed(\Eub G, \cX) = \Cz(\Eub G \setminus \cX)\).
  Applying the crossed product functor \(\rtimes_\mu G\) to the latter equality immediately implies the claim. 
\end{proof}

Combining this with the boundary map on the K-theory of the stable Higson compactification associated to the pair \((\Eub G, \mathcal{X})\), we obtain a canonical map
\begin{equation}
  \delta_{\mathcal{X}} \colon \K_{\ast+1}(\sHigComRed(\mathcal{X}) \rtimes_\mu G) \xrightarrow{\partial} \K_*(\sHigComRed(\Eub G,\mathcal{X})\rtimes_\mu G) \cong \K_G^{-\ast}(\Eub G, \mathcal{X}). \label{eq:HigCom_Ktheory_comparison}
  \end{equation}
In order to obtain numerical invariants on the analytic structure groups from \labelcref{eq:numerical_pairing_structure_group}, it is thus relevant to answer the question if a given topological K-theory class \(\theta \in \K_G^{-\ast}(\Eub G, \mathcal{X})\) lifts to an element \(\tilde{\theta} \in \K_{\ast+1}(\sHigComRed(\mathcal{X}) \rtimes_\mu G)\) along \(\delta_\cX\).
We have the following abstract criterion.
\begin{prop}
  Let \(\theta \in \K_G^{n+1}(\Eub G, \mathcal{X})\).
  Then \(\theta\) lies in the image of \(\delta_\cX\) if and only if its image under the natural map \(\K_G^{n+1}(\Eub G, \mathcal{X}) \to \K_{G}^{n+1}(\Eub G)\) lies in the image of the equivariant co-assembly map \(\K_{-n}(\sHigCorRed(\Eub G) \rtimes_\mu G) \to \K^{n+1}_G(\Eub G)\).
  
\end{prop}
\begin{proof}
  The statement follows from the diagram %
  \[
    \begin{tikzcd}
      \K_{-n}(\sHigCorRed(\Eub G) \rtimes_\mu G) \rar  & \K^{n+1}_G(\Eub G) \rar  & \K_{-n-1}(\sHigComRed(\Eub G) \rtimes_\mu G) \\
      0 \rar \uar & \K_G^{n+1}(\Eub G, \mathcal{X}) \rar["\cong"] \uar & \K_{-n-1}(\sHigComRed(\Eub G, \mathcal{X}) \rtimes_\mu G) \uar \\
      \K_{1-n}(\sHigCorRed(\mathcal{X}) \rtimes_\mu G) \rar \uar & \K_G^{n}(\mathcal{X}) \rar \uar & \K_{-n}(\sHigComRed(\mathcal{X}) \rtimes_\mu G)  \ar[ul, "\delta_\cX", dotted] \uar
    \end{tikzcd}
  \]
  with exact rows and columns.
\end{proof}

In particular, if the equivariant co-assembly map is surjective, then all relative K-theory classes lift to invariants on the structure group.
If we know that \(G\) admits a \(\gamma\)-element, we know even more.

\begin{prop} \label{prop:dual_boundary_surj}
  Let \(\cX\) be \(\sigma\)-\(G\)-compact, \(\rtimes_{\mu}\) be an exact crossed product functor and assume that \(G\) admits a \(G\)-finite \(\Eub G\) and a \(\gamma\)-element.
  Then the map \(\delta_\cX\) of \labelcref{eq:HigCom_Ktheory_comparison} is split-surjective.
\end{prop}
\begin{proof}
By exactness of the K-theory sequence of the stable Higson compactification associated to the pair \((\Eub G, \cX)\), it is enough to show that the map \(\K_{\ast+1}(\sHigComRed(\Eub G) \rtimes_\mu G) \to \K_{\ast+1}(\sHigComRed \cX \rtimes_\mu G)\) is split-injective.
Indeed, by \cite[Corollary~5.3]{EngelWulffZeidler},  the equivariant coarse co-assembly map \(\K_{\ast}(\sHigCorRed(\Eub G) \rtimes_\mu G) \xrightarrow{\partial} \K^{1-\ast}_G(\Eub G)\) is split-surjective.
Then, in turn by exactness of the dual Higson--Roe sequence, the map \(\K_\ast(\sHigComRed(\Eub G) \rtimes_\mu G) \to \K_\ast(\sHigCorRed(\Eub G) \rtimes_\mu G)\) is split-injective.
Then choose a retraction \(r \colon \K_\ast(\sHigCorRed(\Eub G) \rtimes_\mu G) \to \K_{\ast}(\sHigComRed(\Eub G) \rtimes_\mu G)\).
Now the following commutative diagram
\[
  \begin{tikzcd}
    \K_\ast(\sHigComRed(\Eub G) \rtimes_\mu G) \ar[r] \ar[d,tail] & \K_\ast(\sHigComRed(\cX) \rtimes_\mu G) \ar[d] \ar[l, bend left=10, dashed] \\
    \K_\ast(\sHigCorRed(\Eub G) \rtimes_\mu G) \ar[r, "\cong"] \ar[u, bend right, "r"'] & \K_\ast(\sHigCorRed(\cX) \rtimes_\mu G)
  \end{tikzcd}
  \]
  shows that the top horizontal map also admits a retraction. Thus it is split-injective as required.
\end{proof}

A parallel discussion applies to the analytic structure group itself.

\begin{lem}\label{lem:canonical_relative_iso}
  Let \(\cX\) be \(\sigma\)-\(G\)-compact.
  Then the canonical map
 \[
\Strg^G_\ast(\Eub G, \cX) \to \K_\ast^G(\Eub G, \cX)
  \]
  is an isomorphism.
\end{lem}
\begin{proof}
As in the proof of \Cref{lem:canonical_relative_iso_dual}, let $\cX=\bigcup_{n\in\N}X_n$ and $\Eub G=\bigcup_{n\in\N}Y_n$ be the filtrations, so that the inclusions $X_n=\cX\cap Y_n\subset Y_n$ are coarse equivalences. Therefore, \(\Roe(X_n \subset Y_y) = \Roe(Y_n)\) and \(\Roe(Y_n,X_n) = 0\).
  In particular, we also have \(\Roe[G](X_n \subset Y_n) = \Roe[G](Y_n)\) and so \(\Roe[G](Y_n,X_n) = 0\).
  Thus the claim follows from exactness of the equivariant relative Higson--Roe sequence for $\sigma$-spaces defined in \Cref{defn:HigsonRoeRelativeEquivariantSigma} by taking direct limits.
\end{proof}

\begin{prop}\label{prop:boundary_inj}
Let \(\cX\) be \(\sigma\)-\(G\)-compact and assume that \(G\) admits a \(\gamma\)-element.
  Then the boundary map
  \[
    \partial_\cX \colon \K_{\ast+1}^G(\Eub G, \cX) \cong \Strg^G_{\ast+1}(\Eub G, \cX) \xrightarrow{\partial} \Strg_{\ast}^G(\cX)
    \]
    is split-injective, where the first map is the isomorphism of \cref{lem:canonical_relative_iso}.
\end{prop}
\begin{proof}
  By exactness of the structure group exact sequence for the pair \((\Eub G, \cX)\), it is enough to show that the map \(\Strg_\ast^G(\cX) \to \Strg^G_\ast(\Eub G)\) is split-surjective.
  By the existence of a \(\gamma\)-element, the assembly map \(\K_\ast^G(\Eub G) \to \K_\ast(\Roe[G](\Eub G))\) is split-injective and the boundary map \(\partial \colon \K_{\ast+1}(\Roe[G](\Eub G)) \to \Strg^G_\ast(\Eub G)\) split-surjective.
  Choose a section \(s \colon \Strg^G_\ast(\Eub G) \to \K_{\ast+1}(\Roe[G](\Eub G))\).
  Then the following diagram 
  \[
    \begin{tikzcd}
      \K_{\ast+1}(\Roe[G](\cX)) \rar["\cong"] \dar["\partial"] & \K_{\ast+1}(\Roe[G](\Eub G)) \dar[two heads, "\partial"] \\
      \Strg^G_\ast(\cX) \rar & \Strg^G_\ast(\Eub G) \uar["s", bend left] \lar[bend right=20, dashed]
    \end{tikzcd}
  \]
  shows that the bottom horizontal map also admits a section, as desired.
\end{proof}

\begin{rem}
  If, in addition, \(G\) is exact and satisfies the Baum--Connes conjecture, then the maps in \cref{prop:dual_boundary_surj,prop:boundary_inj} are both isomorphisms (for \(\mu = \mathrm{red}\), compare \cite[Corollary~5.4]{EngelWulffZeidler}).
\end{rem}

\begin{thm} Let \(\mathcal{X}\) be a \(\sigma\)-\(G\)-compact \(\sigma\)-\(G\)-metric space.
  Then we have a diagram commuting up to a sign $(-1)^n$ and relating the pairing between the structure group and the K-theory of the stable Higson corona as follows
  \begin{equation*}
    \begin{tikzcd}
      \Strg_n^G(\mathcal{X}) \arrow[r, phantom, "\otimes" description]  &
      \K_{-n}(\sHigComRed(\mathcal{X}) \rtimes_\mu G) \dar["\delta_{\mathcal{X}}"] \rar 
      & \Z \arrow[d, equal] \\
       \K_{n+1}^G(\Eub G, \mathcal{X}) \uar["\partial_{\mathcal{X}}"]\arrow[r, phantom, "\otimes" description]  &
      \K^{n+1}_G(\Eub G, \mathcal{X}) \rar &\Z,
    \end{tikzcd}
  \end{equation*}
  that is, for \(\theta \in \K_{-n}(\sHigComRed(\mathcal{X}) \rtimes_\mu G)\) and \(x \in \K_{n+1}^G(\Eub G, \mathcal{X})\), we have 
  \[
    \langle \partial_{\mathcal{X}} x, \theta \rangle = (-1)^n \langle x, \delta_{\mathcal{X}} \theta \rangle.
  \]
  
  Moreover, if \(G\) admits a \(G\)-finite \(\Eub G\) and a \(\gamma\)-element, then \(\delta_{\mathcal{X}}\) and \(\partial_{\mathcal{X}}\) are split-surjective and split-injective, respectively.
  \label{thm:gamma_element_rich_pairing}
\end{thm}

\begin{proof}
From the $\sigma$-version of two diagrams in \Cref{thm:slantproductpropertiessummarized} with \((\Eub G,\cX)\) in place of \((Y,B)\) and \((\{*\},\emptyset)\) in place of \((X,A)\) we obtain a diagram
  \begin{equation*}
    \begin{tikzcd}
      \Strg_n^G(\mathcal{X}) \arrow[r, phantom, "\otimes" description]  &
      \K_{-n}(\sHigComRed(\mathcal{X}) \rtimes_\mu G) \dar["\partial"] \rar 
      & \Z \arrow[d, equal] \\
       \Strg_{n+1}^G(\Eub G, \mathcal{X}) \dar["\cong"]\uar["\partial"]\arrow[r, phantom, "\otimes" description]  &
      \K_{-n-1}(\sHigComRed(\Eub G,\mathcal{X})\rtimes_\mu G) \rar &\Z \arrow[d, equal] \\
       \K_{n+1}^G(\Eub G, \mathcal{X}) \arrow[r, phantom, "\otimes" description]  &
      \K^{n+1}_G(\Eub G, \mathcal{X})\uar["\cong"] \rar &\Z
    \end{tikzcd}
  \end{equation*}
whose upper part commutes up to a sign $(-1)^n$ and whose lower part commutes.

The second claim was already part of \Cref{prop:dual_boundary_surj} \Cref{prop:boundary_inj}. 
\end{proof}

\section{An application to positive scalar curvature metrics}
\label{sec:PSC_application}
In this final section, we discuss how our pairings can be applied to secondary invariants associated to  positive scalar curvature metrics.

To this end, let \(X\) be a space and consider Stolz' positive scalar curvature bordism sequence
\[
  \dotsc \to \StolzRel_{n+1}(X) \xrightarrow{\partial} \StolzPos_n(X) \to \SpinBordism_n(X) \to \StolzRel_{n}(X) \to \dots.
\]
Here, briefly, \(\StolzRel_{n+1}(X)\) consists of suitable bordism classes of compact spin \((n+1)\)-manifolds \(W\) with maps \(W \to X\) together with a psc metric on \(\partial W\).
The group \(\StolzPos_n(X)\) are psc bordisms classes of closed spin \(n\)-manifolds endowed with a psc metric.
Moreover, \(\SpinBordism_n(X)\) denotes the usual spin bordism group.
The maps are all given by the canonical forgetful maps.
For more details, we refer to~\cite{Stolz98Concordance}, see also~\cite{PiazzaSchick:StolzPSC,XieYuPscLocAlg}.
In work of \textcite{PiazzaSchick:StolzPSC}, \textcite{XieYuPscLocAlg} as well as the second author~\cite{zeidler_secondary,}, it was established that for a discrete group \(G\) there exists a map of the Stolz sequence to the Higson--Roe sequence yielding the following commutative diagram
\begin{equation}\label{mapping_psc_to_HR}
  \begin{tikzcd}
   \StolzRel_{n+1}(\Bfree G) \rar["\partial"] \dar["\alpha"] & \StolzPos_n(\Bfree G) \rar \dar["\rho"] & \SpinBordism_n(\Bfree G) \dar \rar["\alpha"] & \StolzRel_{n}(\Bfree G) \dar \\
   \KO_{n+1}(\Roe[G](\Efree G)) \rar["\partial_{\mathrm{HR}}"] & \StrgReal_{n}^G(\Efree G) \rar & \KO_{n}(\Bfree G) \rar & \KO_{n}(\Roe[G](\Efree G)),
  \end{tikzcd}
\end{equation}
where the second line is the real K-theory version of the Higson--Roe sequence.
Our constructions of pairings and slant products introduced in the previous sections apply without essential changes to the setting of real K-theory as all we have done is to apply the existence of various long exact sequences (in particular, we have not essentially relied on \(2\)-periodicity in the complex case or used Künneth theorems which are more subtle in the real setting).

We will now focus on the higher rho-invariant \(\rho \colon \StolzPos_n(\Bfree G) \to \StrgReal^G_{n}(\Efree G)\) which is a bordism invariant of psc metrics.
There has been considerable interest in proving detection results and extract numerical invariants out of psc metrics via the higher rho invariant, see \textcite{piazza2025mappinganalyticsurgeryhomology} as well as \textcite{WZY:rho_numbers}.
Combining the rho-invariant with (the real version of) our pairing
\[
  \langle\blank,\blank\rangle\colon \StrgReal^G_n(\Efree G) \otimes \KO_{-n}(\sHigComRed (\Efree G)\rtimes_\mu G)\to \KO_{0}(\Roe\{*\}) = \Z,
\]
means that the group \(\KO_{-n}(\sHigComRed (\Efree G)\rtimes_\mu G)\) can be viewed as an abstract repository of numerical rho-invariants in this spirit.
That is, we have a homomorphism
\begin{align}
  \KO_{-n}(\sHigComRed(\Efree G) \rtimes_\mu G) &\to \Hom(\StolzPos_n(\Bfree G), \Z) \label{eq:psc_pairing} \\
  \theta &\mapsto \left( [M,g] \mapsto \langle\rho([M,g]), \theta\rangle \right).\nonumber
\end{align}

In the rest of this section, we discuss that a rephrasing of the detection results of \cite{XYZ} imply that \labelcref{eq:psc_pairing} is quite rich provided that we are in the realm of our \cref{thm:gamma_element_rich_pairing}.
Note that there is a conceptual mismatch between the two rows in \labelcref{mapping_psc_to_HR} because the lower row is \(8\)-periodic in \(n\), whereas the upper row is not (for instance, it does not exist in negative dimensions).
To formulate the result lightly it is thus convenient to artificially enforce Bott periodicity on the upper row.
To this end, let \(\BottMfd\) be a Bott manifold, that is, an \(8\)-dimensional simply-connected spin manifold with \(\hat{\mathrm{A}}(\BottMfd) = 1\).
For any of the constituents of Stolz' sequence \(\mathcal{S} \in \{\StolzRel, \StolzPos, \SpinBordism\}\) it makes sense to define
\[
  \mathcal{S}_{n}(X)[\BottMfd^{-1}] \coloneqq \varinjlim \left( \mathcal{S}_{n}(X) \xrightarrow{\times \BottMfd} \mathcal{S}_{n+8}(X) \xrightarrow{\times \BottMfd} \dotsc \right)
\]
turning the upper sequence \(8\)-periodic. 
The maps from \labelcref{mapping_psc_to_HR} then extend to the Bott-stabilized versions of Stolz' sequence by Bott-periodicity in real K-theory.

\begin{prop} \label{prop:enough_psc_metrics}
  The image of the map \(\partial_{\Efree G} \colon \KO_{n+1}^G(\Eub G, \Efree G) \to \StrgReal^G_n(\Efree G)\) is rationally contained in the image of the stable \(\rho\)-invariant of positive scalar curvature metrics \(
    \rho \colon \StolzPos_n(\Bfree G)[\BottMfd^{-1}] \to \StrgReal^G_{n}(\Efree G).
  \)
\end{prop}
\begin{proof}
  Consider the canonical map \(\pi \colon  \KO_{n+1}^G(\Eub G) \to \KO_{n+1}^{G}(\Eub G, \Efree G)\) as well as the assembly map associated to \(\Eub G\) in the form \(\tilde{\mu} \colon  \KO_{n+1}^G(\Eub G) \to \KO_{n+1}(\Roe[G](\Eub G)) \xleftarrow{\cong} \KO_{n+1}(\Roe[G](\Efree G))\), where we implicitly apply the fact that \(\Efree G\) and \(\Eub G\) are equivalent on the level of equivariant Roe algebras.

  Then a diagram chase using exactness shows that
  \begin{align}
    \im(\partial_{\Efree G} \circ \pi) &= \ker\left(\StrgReal_n^G(\Efree G) \to \StrgReal^G_n(\Eub G)\right) \cap \ker\left(\StrgReal_n^G(\Efree G) \to \KO_n^G(\Efree G)\right) \nonumber
    \\&= \im(\partial_{\mathrm{HR}} \circ \tilde{\mu}),
    \label{eq:higsonRoe_boundary_same_image}
  \end{align}
  where \(\partial_{\mathrm{HR}} \colon \KO_{n+1}(\Roe[G](\Efree G)) \to \StrgReal^G_n(\Efree G)\) is the boundary map in the Higson--Roe sequence associated to the space \(\Efree G\).
  
  It is a consequence of \cite[Theorem~3.8 and Proposition~2.6]{XYZ} that
  \begin{equation}
    \im((\StolzRel_{n+1}(\Bfree G)[\BottMfd^{-1}] \xrightarrow{\alpha} \KO_{n+1}(\Roe[G](\Efree G))) \otimes \C) \supseteq \im(\tilde{\mu} \otimes \C). \label{eq:StolzRel_assembly_inclusion}
  \end{equation}
  
  In light of the delocalized APS index theorem (see \cite{XieYuPscLocAlg,PiazzaSchick:StolzPSC,zeidler_secondary}), the following diagram is commutative
  \[
    \begin{tikzcd}
      \StolzRel_{n+1}(\Bfree G)[\BottMfd^{-1}] \ar[r, "\alpha"] \ar[d, "\partial"] & \KO_{n+1}(\Roe[G](\Efree G)) \ar[d, "\partial_{\mathrm{HR}}"]\\
      \StolzPos_{n}(\Bfree G)[\BottMfd^{-1}] \ar[r, "\rho"] & \StrgReal_n^G(\Efree G).
    \end{tikzcd}
  \]
  In sum, we obtain
  \[
    \im(\rho \otimes \C) \supseteq \im((\partial_{\mathrm{HR}} \circ \alpha) \otimes \C) \underset{\labelcref{eq:StolzRel_assembly_inclusion}}{\supseteq} \im((\partial_{\mathrm{HR}} \circ \tilde{\mu}) \otimes \C)  \underset{\labelcref{eq:higsonRoe_boundary_same_image}}{=} \im((\partial_{\Efree G} \circ \pi) \otimes \C)
  \]
  Since the map \(\pi \colon  \KO_{n+1}^G(\Eub G) \to \KO_{n+1}^{G}(\Eub G, \Efree G)\) is rationally surjective (e.g.\ a consequence of the equivariant Chern character \cite{Lueck:Chern,Baum-Connes:Chern}), this implies the desired claim \(\im(\rho \otimes \C) \supseteq \im(\partial_{\Efree G}\otimes \C)\).
\end{proof}

Combining (a real version) of \cref{thm:gamma_element_rich_pairing} with \cref{prop:enough_psc_metrics} immediately yields the following statement which is dual to the \enquote{Novikov rho invariant} of \cite[Section~7]{WZY:rho_numbers} and \cite[Chapter~15]{piazza2025mappinganalyticsurgeryhomology}.

\begin{cor}
  Suppose that \(G\) admits a \(G\)-finite \(\Eub G\) and a \(\gamma\)-element. Then there is a composite map 
  \begin{align*}
   \mathcal{R} \colon \KO^{n+1}_G(\Eub G, \Efree G) \otimes \C &\dashrightarrow \KO_{-n}(\sHigComRed(\Efree G) \rtimes_\mu G) \otimes \C 
   \\&\to \Hom(\StolzPos_n(\Bfree G)[\BottMfd^{-1}], \C)\,, 
  \end{align*}
  where the second map is given by \labelcref{eq:psc_pairing} and the first map depends on the \(\gamma\)-element, such that the following holds:
  For any \(x \in \KO_{n+1}^G(\Eub G, \Efree G) \otimes \C\), there exists \(z \in \StolzPos_n(\Bfree G)[\BottMfd^{-1}] \otimes \C\) such that for all \(\vartheta \in \KO^{n+1}_{G}(\Eub G, \Efree G)\otimes \C\) we have
  \[
    \mathcal{R}_\vartheta(z) = \langle x, \vartheta \rangle.
  \]
\end{cor}
\begin{proof}
  Let \(s \colon \KO^{n+1}_G(\Eub G, \Efree G) \otimes \C \dashrightarrow \KO_{-n}(\sHigComRed(\Efree G) \rtimes_\mu G) \otimes \C\) be a split of \(\delta_{\Efree G}\) according to \cref{thm:gamma_element_rich_pairing}.
  Now for any \(x \in \KO_{n+1}^G(\Eub G, \Efree G)\), use \cref{prop:enough_psc_metrics} to find \(z \in \StolzPos_n(\Bfree G)[\BottMfd^{-1}] \otimes \C\) with \(\rho(z) = \partial_{\Efree G}x\).
  Then 
    \[
    \mathcal{R}_\vartheta(z) = \langle \rho(z), s(\vartheta)\rangle = \langle \partial_{\Efree G}x, s(\vartheta)\rangle = \langle x, \delta_{\Efree G} s(\vartheta)\rangle = \langle x, \vartheta\rangle. \qedhere 
  \]
\end{proof}

\printbibliography[heading=bibintoc]
\end{document}